\documentclass[a4paper,12pt]{amsart} 

\usepackage[utf8]{inputenc} 
\usepackage[T1]{fontenc}
\usepackage{lmodern}
\usepackage[english]{babel} 

\usepackage{calrsfs}

\usepackage[
	left=2.5cm,       
	right=2.5cm,      
	top=3.5cm,        
	bottom=3.5cm,     
	heightrounded,  
]{geometry} 

\usepackage{bookmark}
\usepackage{hyperref} 
\usepackage{varioref}

\usepackage{graphicx} 
\usepackage{xcolor} 

\usepackage{amssymb}
\usepackage{mathtools}
\usepackage{mathrsfs} 

\usepackage{nicefrac}

\theoremstyle{plain}
	\newtheorem{theorem}{Theorem}[section]
	\newtheorem{proposition}[theorem]{Proposition}
	\newtheorem{lemma}[theorem]{Lemma}
	\newtheorem{corollary}[theorem]{Corollary}

\theoremstyle{definition}
	\newtheorem{definition}[theorem]{Definition}
	\newtheorem{remark}[theorem]{Remark}

\newcommand{\N}{\mathbb{N}}
\newcommand{\R}{\mathbb{R}}	
\newcommand{\C}{\mathbb{C}}	
\newcommand{\He}{\mathbb{He}}	
\newcommand{\W}{\mathbb{W}}	
\newcommand{\eps}{\varepsilon}
\renewcommand{\phi}{\varphi}
\newcommand{\leb}{\mathcal{L}}

\newcommand{\haush}{\mathcal{S}}
\newcommand{\de}{\partial}

\usepackage{aurical}
\newcommand{\hgt}{\text{\large\Fontamici h}}

\renewcommand{\Im}{\mathrm{Im}}

\DeclareMathOperator{\diverg}{div}
\DeclareMathOperator{\diam}{diam}
\DeclareMathOperator{\Lip}{Lip}
\DeclareMathOperator{\gr}{gr}
\DeclareMathOperator{\spann}{span}

\def\Xint#1{\mathchoice
{\XXint\displaystyle\textstyle{#1}}%
{\XXint\textstyle\scriptstyle{#1}}%
{\XXint\scriptstyle\scriptscriptstyle{#1}}%
{\XXint\scriptscriptstyle\scriptscriptstyle{#1}}%
\!\int}
\def\XXint#1#2#3{{\setbox0=\hbox{$#1{#2#3}{\int}$ }
\vcenter{\hbox{$#2#3$ }}\kern-.6\wd0}}

\def\aint{\Xint-}

\mathchardef\ordinarycolon\mathcode`\:
\mathcode`\:=\string"8000
\begingroup \catcode`\:=\active
  \gdef:{\mathrel{\mathop\ordinarycolon}}
\endgroup

\numberwithin{equation}{section}

\usepackage[initials]{amsrefs}

\begin{document}

\title[Improved Lipschitz Approximation]
{Improved Lipschitz Approximation of $H$-perimeter minimizing boundaries}

\author[R. Monti]{Roberto Monti}
\address{Dipartimento di Matematica, Università degli Studi di Padova, Via Trieste 63, 35121 Padova, Italy}
\email{monti@math.unipd.it} 

\author[G. Stefani]{Giorgio Stefani}
\address{Scuola Normale Superiore, Piazza Cavalieri 7, 56126 Pisa, Italy}
\email{giorgio.stefani@sns.it}

\date{\today}



\begin{abstract}
We prove two new approximation results of $H$-perimeter 
minimizing boundaries by means of intrinsic Lipschitz functions in the setting
of the Heisenberg group $\He^n$ with $n\ge2$. The first one is an improvement
of~\cite{monti14} and is the natural reformulation in $\He^n$ of the classical
Lipschitz approximation in~$\R^n$. The second one is an adaptation of the
approximation via maximal function developed by De Lellis and
Spadaro,~\cite{delellis&spadaro11-1}.   
\end{abstract}

\maketitle

\section{Introduction}

The study of Geometric Measure Theory in the Heisenberg group $\He^n$ started from the pioneering work~\cite{franchietal01}
and the 
regularity of sets that are minimizers for the horizontal perimeter is one of the most important 
open problems in the field. 
The  known regularity results assume some strong \textit{a~priori} regularity
and/or some restrictive geometric structure of the minimizer, 
see~\cites{capognaetal09,capognaetal10,cheng&hwang09,serracassano&vittone14,
monti15}.  On the other hand, examples of minimal surfaces in the first
Heisenberg group~$\He^1$ that are only Lipschitz continuous in the Euclidean
sense have been constructed, see, e.g.,~\cites{pauls06,ritore09}, but no similar
examples of non-smooth minimizers are known in~$\He^n$ with $n\ge2$.

The most natural approach to a regularity theory for $H$-perimeter minimizing sets in the Heisenberg group $\He^n$ is to adapt the classical De Giorgi's regularity theory for perimeter minimizers in $\R^n$.
His ideas have been recently  
improved and generalized by several authors, see
the recent monograph~\cite{maggi12}. In particular, one of the most important
achievements is Almgren's regularity theory of area minimizing integral currents
in $\R^n$ of general codimension,~\cite{almgren00}.  
 For a survey on Almgren's theory and on the long term
program undertaken by
De Lellis and Spadaro to make Almgren's work more readable and exploitable for a
larger community of specialists, we refer to~\cite{ambrosio15} and to the
references therein. For the recent extension of the theory to infinite
dimensional spaces, see~\cite{ambrosioetal15}.

This paper deals with the first step of the regularity theory, namely, 
the Lipschitz approximation.
In fact, in De Giorgi's original approach the approximation is made by 
convolution and the estimates are based on a monotonicity formula. In the
Heisenberg group, however, the validity of a monotonicity formula is not  
clear, see~\cite{daniellietal10}. 
A more flexible approach is the approximation of minimizing boundaries  by means
of Lipschitz graphs, see~\cite{schoen&simon82}. 
 Although the boundary of sets with finite $H$-perimeter is not rectifiable
in
the standard sense and, in fact, may have fractional Hausdorff
dimension,~\cite{kirchheim&serracassano04}, the  notion of
\textit{intrinsic Lipschitz graph} in the sense of~\cite{franchietal06} turns
out to be effective in the approximation, as shown in \cite{monti14}. 

 Here, we prove two new intrinsic Lipschitz approximation 
theorems for $H$-perimeter minimizers in the setting of the Heisenberg
group~$\He^n$ with $n\ge2$.

The first result is an improvement of \cite{monti14} and is the natural reformulation in $\He^n$ of the classical Lipschitz approximation in $\R^n$, see~\cite{maggi12}*{Theorem~23.7}. 
Let  
$\W=\R\times\He^{n-1}$ be the hyperplane  passing through the origin and
orthogonal to the direction~$\nu=-X_1$.
The disk $D_r\subset\W$ centered at the origin is defined using the natural  box norm of~$\He^n$ and the cylinder $C_r(p)$, $p\in\He^n$, is defined as $C_r(p)=p*C_r$, where $C_r=D_r*(-r,r)$.
 We denote by   $\mathbf{e}(E,C_r(p),\nu)$ the  excess of $E$ in 
$C_r(p)$ with respect to the fixed direction $\nu$, that is, the $L^2$-averaged 
oscillation of $\nu_E$, the inner horizontal unit normal to $E$, from the
direction $\nu$ in the cylinder. 
The $2n+1$-dimensional spherical Hausdorff measure $\haush^{2n+1}$  is defined
by the natural distance of~$\He^n$.
Finally, $\nabla^\phi\phi$ is the intrinsic gradient of $\phi$.  We refer the
reader to  Section~\ref{sec:1} for precise definitions.

\begin{theorem}\label{th_intro:lip_app} 
Let $n\geq 2$. There exist positive dimensional constants $C_1(n)$, $\eps_1(n)$
and $\delta_1(n)$ with the following property. If $E\subset\He^n$ is an
$H$-perimeter minimizer in the cylinder $C_{5124}$ with $0\in \partial E$ and
$\mathbf{e}(E,C_{5124},\nu)\le\eps_1(n)$ then, letting
\begin{equation*}
M=C_1\cap\de E, \qquad M_0=\big\{q\in M : \sup_{0<s<256}\mathbf{e}(E,C_s(q),\nu)\le\delta_1(n) \big\}, 
\end{equation*}
there exists an intrinsic Lipschitz function $\phi\colon\W\to\R$ such that
\begin{equation*}
\sup_\W|\phi|\le C_1(n)\,\mathbf{e}(E,C_{5124},\nu)^{\frac{1}{2(2n+1)}},\quad \Lip_H(\phi)\le1, 
\end{equation*}
\begin{equation*}
M_0\subset M\cap\Gamma, \qquad \Gamma=\gr(\phi|_{D_1}),
\end{equation*}
\begin{align*}
\haush^{2n+1}(M\bigtriangleup\Gamma)&\le C_1(n)\,\mathbf{e}(E,C_{5124},\nu),\\
\int_{D_1}|\nabla^\phi\phi|^2\ d\leb^{2n}&\le C_1(n)\,\mathbf{e}(E,C_{5124},\nu).
\end{align*}
\end{theorem}

The Lipschitz approximation proved in \cite{monti14} is limited to the estimate
$\haush^{2n+1}(M\bigtriangleup\Gamma)\le C_1(n)\,\mathbf{e}(E,C_{5124},\nu)$.
Here, we give a more elementary proof of a more general result following  the
scheme outlined in~\cite{maggi12}*{Section~23.3}. The fundamental tool used in
the proof is the height estimate recently established
in~\cite{monti&vittone15}*{Theorem~1.3}.
Theorem~\ref{th_intro:lip_app} holds also for $(\Lambda,r_0)$-minimizers  of
$H$-perimeter, see the more general formulation given in
Theorem~\ref{th:lip_app_easy} of Section~\ref{sec:2}.

Theorem~\ref{th_intro:lip_app} is the starting point for the proof of our 
second result,  which  is obtained  using an adaptation   
to the setting of $H$-perimeter minimizers in $\He^n$ of the ideas developed
in~\cite{delellis&spadaro11-1} by De Lellis and Spadaro for
area minimizing integral currents.

\begin{theorem}\label{th_intro:alpha_app}
Let $n\geq 2$ and $\alpha\in(0,\frac{1}{2})$. There exist positive constants $C_2(n)$, $\eps_2(\alpha,n)$ and $k_2=k_2(n)$ with the following property. For any set  $E\subset\He^n$ that is an	$H$-perimeter minimizer in the cylinder $C_{k_2}$ with $0\in\de E$  and $\mathbf{e}(E,C_{k_2},\nu)\le\eps_2(\alpha,n)$, there exist a set $K\subset D_1$ 
and an intrinsic Lipschitz function $\phi\colon\W\to\R$ such that:
\begin{equation*}
\leb^{2n}(D_1\setminus K)\le C_2(n)\,\mathbf{e}(E,C_{k_2},\nu)^{1-2\alpha}
\end{equation*}
\begin{equation*}
\gr(\phi|_K)=\de E\cap \big(K*(-1,1)\big),\\[1mm]
\end{equation*}
\begin{equation*}
\Lip_H(\phi)\le C_2(n)\,\mathbf{e}(E,C_{k_2},\nu)^\alpha,\\[1mm]
\end{equation*}
\begin{equation*}
\haush^{2n+1}\big((\de E\bigtriangleup\gr(\phi))\cap C_1\big)\le C_2(n)\,\mathbf{e}(E,C_{k_2},\nu)^{1-2\alpha}, \\[1mm]
\end{equation*}
\begin{equation*}
\int_{D_1}|\nabla^\phi\phi|^2\ d\leb^{2n}\le C_2(n)\,\mathbf{e}(E,C_{k_2},\nu).
\end{equation*}
\end{theorem}

Theorem~\ref{th_intro:alpha_app} holds also for $(\Lambda,r_0)$-minimizers  of
$H$-perimeter, see the more general formulation of this result given in
Corollary~\ref{coroll:alpha_app_easy} of Section~\ref{sec:3}.

The first step in~\cite{delellis&spadaro11-1}  is to
establish a so-called $BV$\unkern\emph{estimate} on the vertical slices of the
area minimizing integral current,
see~\cite{delellis&spadaro11-1}*{Lemma~A.1}. The proof of this estimate  uses
several fundamental results of the theory of integral currents in $\R^n$. Thus
far,  a  theory for integral currents in $\He^n$ is not yet well established,
see~\cite{franchietal07}, and a similar estimate for the slices of the boundary
of an $H$-perimeter minimizer is not clear. 
However,  when  the minimizer is  the intrinsic epigraph of an intrinsic
Lipschitz function, the $BV$\unkern\emph{estimate} is an easy consequence of the
Cauchy--Schwarz inequality and of the area formula. Therefore,  when  $E$ is
an $H$-perimeter minimizer, we can overcome the problem with the
following trick: first, by Theorem~\ref{th_intro:lip_app}, we approximate the
boundary of $E$ with the intrinsic graph of a suitable intrinsic Lipschitz
function; second, up to an error which is comparable to the excess, we replace
the $BV$\unkern\emph{estimate} on the slices of the boundary of $E$ with the
$BV$\unkern\emph{estimate} on the slices of the approximating graph. 
A fundamental tool used in our argument is the  Poincaré 
inequality recently established in~\cite{cittietal16}.

In the case of minimizing integral currents in $\R^n$, the Lipschitz approximation in the spirit
of Theorem \ref{th_intro:alpha_app} is the starting point of the so-called harmonic approximation, that gives the decay estimates for excess. In the setting of $\mathbb H^n$, deriving the harmonic approximation from Theorem \ref{th_intro:alpha_app} is still an open problem, see \cite{monti15}.

\section{Preliminaries}\label{sec:1}

In this section, we fix the notation and recall
some  basic facts on intrinsic Lipschitz functions, on the area formula,
and on the height bound for $H$-perimeter minimizers. 
The reader familiar with these results can skip this section.

\subsection{The Heisenberg group}

The $n$-th Heisenberg group is the manifold $\He^n=\C^n\times\R$ 
endowed with the group law $(z,t)*(w,s)=(z+w,t+s+P(z,w))$ for $(z,t),(w,s)\in\He^n$, 
where $z,w\in\C^n$, $t,s\in\R$ and $P\colon\C^n\times\C^n\to\R$ is the bilinear form
\begin{equation*}
P(z,w)=2\,\Im\bigg(\sum_{j=1}^nz_j\bar{w}_j\bigg), \quad z,w\in\C^n.
\end{equation*}
The \emph{left translations} $\tau_q\colon\He^n\to\He^n$ are defined by
$\tau_q(p)=q*p$. 
The automorphisms $\delta_\lambda\colon\He^n\to\He^n$, $\lambda>0$, of the form
\begin{equation*}
\delta_\lambda (z,t)=(\lambda z,\lambda^2 t), \quad (z,t)\in\He^n,
\end{equation*}
are called \emph{dilations}. We use the abbreviations  $\lambda
p=\delta_\lambda(p)$ and $\lambda E=\delta_\lambda(E)$ for $p\in\He^n$ and
$E\subset\He^n$.

For any $p=(z,t)\in\He^n$,  let $\|p\|_\infty=\max\{ |z|,|t|^{1/2}\}$ be the \emph{box norm}. It satisfies the triangle inequality 
 \begin{equation*}
\|p*q\|_\infty\le \|p\|_\infty+\|q\|_\infty, \quad  p,q\in\He^n.
\end{equation*} 
The function $d_\infty\colon\He^n\times\He^n\to[0,\infty)$, $d(p,q)=\|p^{-1}*q\|$ 
for  $p,q\in\He^n$, is a left invariant distance  on $\He^n$ equivalent to the
Carnot-Carathéodory  distance. We define the open ball
centered at $p\in\He^n$ and with radius $r>0$ as
\begin{equation}\label{eq:def_H_box_ball}
B_r(p)=\{q\in\He^n : d_\infty(q,p)<r\}=p*\{q\in\He^n : \|q\|_\infty<r\}.
\end{equation}  
In the case $p=0$, we let $B_r=B_r(0)$.

For any $s\ge0$, we denote by $\haush^s$ the spherical Hausdorff measure in $\He^n$ constructed with the left invariant metric $d_\infty$. Namely, for any $E\subset\He^n$ we let
\begin{equation*}
\haush^s(E)=\lim_{\delta\to0}\haush^s_\delta(E),
\end{equation*}
where
\begin{equation*}
\haush^s_\delta(E)=\inf\Big\{
\sum_{n\in\N}(\diam B_i)^s : E\subset\bigcup_{n\in\N}B_i,\ B_i\text{ balls as in~\eqref{eq:def_H_box_ball}},\ \diam B_i<\delta\Big\}
\end{equation*}
and $\diam$ is the diameter in the distance $d_\infty$. The correct dimension to measure hypersurfaces is $s=2n+1$.

We identify an element $z=x+iy\in\C^n$ with $(x,y)\in\R^{2n}$. The Lie algebra of left invariant vector fields in $\He^n$ is spanned by the vector fields
\begin{equation}\label{eq:vector_fileds_X_Y_T}
X_j=\frac{\de}{\de x_j}+2y_j\frac{\de}{\de t},\quad Y_j=\frac{\de}{\de y_j}-2x_j\frac{\de}{\de t},\quad T=\dfrac{\de}{\de t}, \quad j=1,\dots,n.
\end{equation} 
We denote by $H$ the \emph{horizontal sub-bundle} of $T\He^n$. Namely, for any $p=(z,t)\in\He^n$, we let
\begin{equation*}
H_p=\spann\big\{X_1(p),\dots,X_n(p),Y_1(p),\dots,Y_n(p)\big\}.
\end{equation*}

Let $g$ be the left invariant Riemannian  metric on $\He^n$ that makes
orthonormal the vector fields $X_1,\dots,X_n,Y_1,\dots,Y_n,T$. The metric $g$
induces a volume form on $\He^n$ that is left invariant and coincides with the
Lebesgue measure $\leb^{2n+1}$. For tangent vectors $V,W\in T\He^n$, we let
\begin{equation*}
\langle V,W\rangle_g=g(V,W) \quad \text{ and } \quad |V|_g=g(V,V)^{1/2}.
\end{equation*}

Let $\Omega\subset\He^n$ be an open set. A \emph{horizontal section}  $V\in
C^1_c(\Omega;H)$ is a vector field of the form
\begin{equation*}
V=\sum_{j=1}^nV_jX_j+V_{j+n}Y_j,
\end{equation*}
where $V_j\in C^1_c(\Omega)$ for any $j=1,\dots,2n$.  The sup-norm with respect
to $g$ of a horizontal section $V\in C^1_c(\Omega;H)$ is
\begin{equation*}
\|V\|_g=\max_{p\in\Omega}|V(p)|_g.
\end{equation*}  
The \emph{horizontal divergence} of $V$ is
\begin{equation*}
\diverg_H V=\sum_{j=1}^nX_jV_j+Y_jV_{j+n}.
\end{equation*}

\subsection{Locally finite perimeter sets}

A $\leb^{2n+1}$-measurable set $E\subset\He^n$ has 
\emph{finite $H$-perimeter}  in an open set $\Omega\subset\He^n$ if   
\begin{equation*}
P_H(E;\Omega)=\sup\Big\{\int_E\diverg_H V\, d\leb^{2n+1} :\ V\in C^1_c(\Omega;H),\ \|V\|_g\le1\Big\}<\infty.
\end{equation*} 
If $P_H(E;A)<\infty$ for any open set $A\subset\subset \Omega$,  we say that $E$
has  \emph{locally finite $H$-perimeter}  in
$\Omega$. In this case, the  mapping
$A\mapsto P_H(E;A) = \mu_E(A)$ extends from open sets to a Radon measure $\mu_E$
on $\Omega$. By the Radon-Nykodim Theorem, there exists a $\mu_E$-measurable
function $\nu_E\colon\Omega\to H$ such that $|\nu_E|_g=1$ $\mu_E$-a.e., and  the
\emph{Gauss--Green formula}
\begin{equation*}
\int_E \diverg_H V\ d\leb^{2n+1}=-\int_\Omega \langle V,\nu_E\rangle_g\ d\mu_E
\end{equation*}
holds for any $V\in C^1_c(\Omega;H)$. 
We call $\nu_E$ the \emph{horizontal inner normal} of $E$ in $\Omega$.
The \emph{measure theoretic boundary} of a $\leb^{2n+1}$-measurable set $E\subset\He^n$ is the set
\begin{equation*}
\de E=\big\{p\in\He^n : \leb^{2n+1}(E\cap B_r(p))>0 \text{ and } \leb^{2n+1}(B_r(p)\setminus E)>0 \text{ for all } r>0\big\}.
\end{equation*}
Let $E$ be a set with locally finite $H$-perimeter  in $\He^n$.
 Then the measure $\mu_E$ is 
concentrated on $\de E$ and, actually, on a subset $\de^*E\subset\de E$ called
the \emph{reduced boundary} of~$E$. This follows from the structure theorem for
sets with locally finite $H$-perimeter, see~\cite{franchietal01}. Moreover, up
to modifying $E$ on a Lebesgue negligible set, one can always assume that $\de
E$ coincides with the topological boundary of $E$,
see~\cite{serracassano&vittone14}*{Proposition~2.5}. 

\subsection{Perimeter minimizers}

Let $\Omega\subset\He^n$ be an open set and let $E$  be a set with locally
finite $H$-perimeter in $\He^n$. We say that the set $E$ is a
\emph{$(\Lambda,r_0)$-minimizer of $H$-perimeter} in~$\Omega$ if there exist two
constants $\Lambda\in[0,\infty)$ and $r_0\in(0,\infty]$ such that
\begin{equation*}
P(E;B_r(p))\le P(F;B_r(p))+\Lambda\,\leb^{2n+1}(E\bigtriangleup F)
\end{equation*}
for any measurable set $F\subset\He^n$, $p\in\Omega$ and $r<r_0$  such that
$E\bigtriangleup F\subset\subset B_r(p)\subset\subset\Omega$.

When $\Lambda=0$ and $r_0=\infty$, we say that the set $E$  is a \emph{locally
$H$-perimeter minimizer} in~$\Omega$, that is, we have
\begin{equation*}
P(E;B_r(p))\le P(F;B_r(p))
\end{equation*}
for any measurable set $F\subset\He^n$, $p\in\Omega$  and $r>0$ such that
$E\bigtriangleup F\subset\subset B_r(p)\subset\subset\Omega$.
 
If $E$ is a $(\Lambda,r_0)$-minimizer of $H$-perimeter in~$\Omega$,  then
the difference $\de E\setminus\de^*E$ is $\haush^{2n+1}$-negligible in $\Omega$,
see~\cite{monti&vittone15}*{Corollary~4.2}. Thus, in the following, up to
modifying $E$ on a Lebesgue negligible set, we will tacitly assume that the
reduced boundary and the topological boundary of $E$ coincide.

\begin{remark}[Scaling of $(\Lambda,r_0)$-minimizer]\label{remark:scaling_perim_min}
If the set $E$ is a $(\Lambda,r_0)$-minimizer of $H$-perimeter 
in the open set $\Omega\subset\He^n$ then, for every $p\in\He^n$ and $r>0$,
 the set $E_{p,r}=\delta_{\frac{1}{r}}(\tau_{p^{-1}}(E))$ is   a
$(\Lambda',r_0')$-minimizer of $H$-perimeter in $\Omega_{p,r}$, where
$\Lambda'=\Lambda r$ and $r_0'=r_0/r$. In particular, the product $\Lambda r_0$
is invariant and thus it is convenient to assume that $\Lambda r_0\le1$, as we
shall always do in the following.  
\end{remark}

\subsection{Cylindrical excess}

The \emph{height function} $\hgt\colon\He^n\to\R$  is the group homomorphism
$\hgt(p)=x_1$,  for   $p=(x,y,t)\in\He^n$. Let $\W$ be the (normal) subgroup
of $\He^n$ given by the kernel of $\hgt$, 
\begin{equation*}
\W:=\ker\hgt=\big\{p\in\He^n : \hgt(p)=0\big\}.
\end{equation*}
The \emph{open disk} in $\W$ of radius $r>0$ centered at the origin is the set 
$D_r=\{w\in\W : \|w\|_\infty<r\}$. For any $p\in\W$, we let $D_r(p)=p*D_r\subset\W$. Note that, for all $p\in\W$ and $r>0$, 
\begin{equation}\label{eq:leb_meas_disk}
\leb^{2n}(D_r(p))=\leb^{2n}(D_r)=\kappa_n r^{2n+1},
\end{equation}
with $\kappa_n=\leb^{2n}(D_1)$.
The \emph{open cylinder} with central section $D_r$ and height $2r$ is the set
\begin{equation*}
C_r=D_r*(-r,r):=\{w*s\mathrm{e}_1\in\He^n : w\in D_r,\ s\in(-r,r)\},
\end{equation*}
where $s\mathrm{e}_1=(s,0,\dots,0)\in\He^n$. For any $p\in\He^n$, we let $C_r(p)=p*C_r$.

Let $\pi\colon\He^n\to\W$ be the projection on $\W$ defined, for any $p\in\He^n$, by the formula
\begin{equation}\label{eq:def_proj_on_W}
p=\pi(p)*\hgt(p)\mathrm{e}_1.
\end{equation} 
By~\eqref{eq:def_proj_on_W}, for any $p\in\He^n$ and $r>0$, we have
\begin{equation*}
p\in C_r \iff \pi(p)\in D_r,\ \hgt(p)\in (-r,r) \iff 
\|\pi(p)\|_\infty<r, \ |\hgt(p)|<r.
\end{equation*}
We thus let $\|\cdot\|_C\colon\He^n\to[0,\infty)$ be the map 
\begin{equation}\label{eq:C_quasi_norm}
\|p\|_C:=\max \{ \|\pi(p)\| _\infty,|\hgt(p)|  \}
\end{equation}
for any $p\in\He^n$, so that $C_r=\{p\in\He^n : \|p\|_C<r\}$. 
The map $\|\cdot\|_C$ is a quasi-norm and, by~\eqref{eq:def_proj_on_W}, we have
\begin{equation}\label{eq:quasi_C_equiv_infty}
\|p\|_C\le 2 \|p\|_\infty, \quad \|p\|_\infty\le 2\|p\|_C \quad  p\in\He^n.
\end{equation}
Let $d_C\colon\He^n\times\He^n\to[0,\infty)$  be the quasi-distance induced by
$\|\cdot\|_C$. By~\eqref{eq:quasi_C_equiv_infty}, the cylinder $C_r(p)$ is
comparable with the ball $B_r(p)$ induced by the box norm for any $p\in\He^n$.
Namely, we have
\begin{equation}\label{eq:D_and_C_comparison}
B_{r/2}(p)\subset C_r(p)\subset B_{2r}(p)\ \text{ for all } p\in\He^n,\ r>0.
\end{equation}
A concept which plays a key role in the regularity  theory of
$(\Lambda,r_0)$-minimizers of $H$-perimeter is the notion of excess.

\begin{definition}[Cylindrical excess]\label{def:excess}
Let $E$ be a set with locally finite $H$-perimeter in $\He^n$. The \emph{cylindrical excess} of $E$ at the point $p\in\de E$, at the scale $r>0$, and with respect to the direction $\nu=-X_1$, is defined as
\begin{equation}\label{eq:def_H_excess}
\mathbf{e}(E,p,r,\nu):=\dfrac{1}{r^{2n+1}}\int_{C_r(p)}\dfrac{|\nu_E-\nu|_g^2}{2}\ d\mu_E
=\dfrac{\delta(n)}{r^{2n+1}}\int_{C_r(p)\cap\de^*E} \big(1-\langle \nu_E,\nu\rangle_g\big)\ d\haush^{2n+1}\nonumber
\end{equation} 
where $\mu_E$ is the Gauss-Green measure of $E$, $\nu_E$ is the horizontal inner normal and the multiplicative constant is 
\mbox{$\delta(n)=\tfrac{2\omega_{2n-1}}{\omega_{2n+1}}$}.
\end{definition}  

\noindent
We refer the reader to~\cite{magnani14} for the problem of the coincidence of
perimeter measure and spherical Hausdorff measures. 

For the sake of brevity, we will  set $\mathbf{e}(p,r)=\mathbf{e}(E,p,r,\nu)$ and, in the case $p=0$, $\mathbf{e}(r)=\mathbf{e}(0,r)$. For the elementary properties of the excess, see~\cite{monti&vittone15}*{Section~3.2}.

\subsection{Height bound}

The following result is a fundamental estimate relating the height of the boundary of a $(\Lambda,r_0)$-minimizer of $H$-perimeter with the cylindrical excess, see~\cite{monti&vittone15}*{Theorem~1.3}.

\begin{theorem}[Height bound]\label{th:H-height_bound}
Let  $n\ge2$. There exist positive dimensional constants $\eps_0(n)$ and $C_0(n)$ with the following property. If $E$ is a $(\Lambda,r_0)$-minimizer of $H$-perimeter in the cylinder $C_{16r_0}$ with
\begin{equation*}
\Lambda r_0\le 1, \quad 0\in\de E, \quad \mathbf{e}(16r_0)\le\eps_0(n),
\end{equation*}
then
\begin{equation}\label{eq:th_H-height_bound_estimate}
\sup\Big\{\frac{|\hgt(p)|}{r_0} : p\in C_{r_0}\cap\de E\Big\}\le C_0(n)\,\mathbf{e}(16r_0)^{\frac{1}{2(2n+1)}}.
\end{equation}
\end{theorem}

\begin{remark}\label{remark:no_height_bound}
The estimate~\eqref{eq:th_H-height_bound_estimate} does not hold when $n=1$. In fact, there are sets $E\subset\He^1$ such that $\mathbf{e}(E,0,r,\nu)=0$ but $\de E$ is not flat in $C_{\eps r}$ for any $\eps>0$, see the conclusion of~\cite{monti14}*{Proposition~3.7}.
\end{remark}

\subsection{Intrinsic Lipschitz functions}

We identify the vertical hyperplane
\begin{equation*}
\mathbb{W}=\He^{n-1}\times\R=\big\{(z,t)\in\He^n : x_1=0\big\}
\end{equation*} 
with $\R^{2n}$ via the coordinates $w=(x_2,\dots,x_n,y_1,\dots,y_n,t)$. The line flow of the vector field $X_1$ starting from the point $(z,t)\in\mathbb{W}$ is the curve  
\begin{equation}\label{eq:gamma_exp_X1}
\gamma(s)=\exp(sX_1)(z,t)=(z+s\mathrm{e}_1,t+2y_1s),\ s\in\R,
\end{equation} 
where $\mathrm{e}_1=(1,0,\dots,0)\in\He^n$ and $z=(x,y)\in\C^n\equiv\R^{2n}$.

Let $W\subset\mathbb{W}$ be a set and let $\phi\colon W\to\R$ be a function. The set
\begin{equation}\label{eq:def_intrinsic_epigraph}
E_\phi=\big\{\exp(sX_1)(w)\in\He^n : s>\phi(w),\ w\in W\big\}
\end{equation} 
is called \emph{intrinsic epigraph of $\phi$} along $X_1$, while the set
\begin{equation*}
\gr(\phi)=\big\{\exp(\phi(w)X_1)(w)\in\He^n : w\in W\big\}
\end{equation*} 
is called \emph{intrinsic graph of $\phi$} along $X_1$.
By~\eqref{eq:gamma_exp_X1}, we easily find the identity
\begin{equation*}
\exp(\phi(w)X_1)(w)=w*\phi(w)\mathrm{e}_1 \quad\text{for any } w\in W,
\end{equation*}
thus the intrinsic graph of $\phi$ is the set
$\gr(\phi)=\{w*\phi(w)\mathrm{e}_1\in\He^n : w\in W\}$.
 The \emph{graph map} of the function $\phi\colon W\to\R$, $W\subset\W$, is the map   $\Phi\colon W\to\He^n$, $\Phi(w)= w*\phi(w)\mathrm{e}_1$, $w\in W$. For any $A\subset W$, we let $\gr(\phi|_A)=\Phi(A)$.

The notion of intrinsic Lipschitz function was introduced in~\cite{franchietal06}*{Definition~3.1}.

\begin{definition}[Intrinsic Lipschitz function]\label{def:intrinsic_lip}
Let $W\subset\mathbb{W}$. A function  $\phi\colon W\to\R$  is \emph{$L$-intrinsic Lipschitz}, with $L\in[0,\infty)$, if for all $ p,q\in\gr(\phi)$ we have 
\begin{equation}\label{eq:def_intr_lip}
|\phi(\pi(p))-\phi(\pi(q))|\le L\|\pi(q^{-1}*p)\|_\infty .
\end{equation}
\end{definition}

The definition can be equivalently given in terms of intrinsic cones.
We denote by $\Lip_H(W)$  the set  of  intrinsic Lipschitz functions on the set
$W\subset\W$. If $\phi\in\Lip_H(W)$, we denote by $\Lip_H(\phi)$  the intrinsic
Lipschitz constant of $\phi$,  with no reference to the set
 if  no confusion arises.

An extension theorem for intrinsic Lipschitz functions  was proved for the first
time in~\cite{franchietal11}*{Theorem~4.25}. The following result gives an
explicit estimate of the Lipschitz constant of the extension. The first part is
proved in~\cite{monti14}*{Proposition~4.8}, while the second part follows from
an easy modification of the proof of the first one.

\begin{proposition}\label{prop:intrinsic_lip_ext}
Let $W\subset\W$ and let $\phi\colon W\to\R$ be an $L$-intrinsic  Lipschitz
function. There exists an $M$-intrinsic Lipschitz function $\psi\colon\W\to\R$
with
\begin{equation}\label{eq:lip_const_extension}
M=\left(\sqrt{1+\frac{1}{L+2L^2}}-1\right)^{-2}
\end{equation}
such that $\psi(w)=\phi(w)$ for all $w\in W$.  If $\phi$ is bounded
then there exists an  extension that also satisfies  
$\|\psi\|_{L^\infty(\W)}=\|\phi\|_{L^\infty(W)}$.
\end{proposition}

\noindent
Note that, in~\eqref{eq:lip_const_extension}, we have $M\le 2L$ for all $L\le 0,07$.

We now introduce a non-linear gradient for functions $\phi\colon W\to\R$ with $W\subset\mathbb{W}$ an open set. Let $\mathscr{B}\colon\Lip_{loc}(W)\to L^\infty_{loc}(W)$ be the Burgers' operator defined by
\begin{equation*}
\mathscr{B}\phi=\dfrac{\de\phi}{\de y_1}-4\phi\dfrac{\de\phi}{\de t}.
\end{equation*}
When $\phi\in C(W)$ is only continuous, we say that $\mathscr{B}\phi$ exists in the sense of distributions and is represented by a locally bounded function if there exists a function $\vartheta\in L^\infty_{loc}(W)$ such that
\begin{equation*}
\int_W\vartheta\psi\ dw=-\int_W\left\{\phi\dfrac{\de\psi}{\de y_1}-2\phi^2\dfrac{\de\psi}{\de t}\right\}\ dw
\end{equation*} 
for any $\psi\in C^1_c(W)$. In this case, we let $\mathscr{B}\phi=\vartheta$.

Note that the vector fields $X_2,\dots,X_n,Y_2,\dots,Y_n$ can be naturally restricted to $\mathbb{W}$ and that they are self-adjoint.

Let $\phi\colon W\to\R$ be a continuous function on the open set $W\subset\W$. We say that the intrinsic gradient $\nabla^\phi\phi\in L^\infty_{loc}(W;\R^{2n-1})$ exists in the sense of distributions if the distributional derivatives $X_i\phi$, $\mathscr{B}\phi$ and $Y_i\phi$, with $i=2,\dots,n$, are represented by locally bounded functions in $W$. In this case, we let
\begin{equation}\label{eq:def_intrinsic_grad}
\nabla^\phi\phi=(X_2\phi,\dots,X_n\phi,\mathscr{B}\phi,Y_2\phi,\dots,Y_n\phi),
\end{equation}
and we call $\nabla^\phi\phi$ the \emph{intrinsic gradient of $\phi$}. When $n=1$, the intrinsic gradient reduces to $\nabla^\phi\phi=\mathscr{B}\phi$.

The intrinsic gradient~\eqref{eq:def_intrinsic_grad}  has a strong non-linear
character. This partially motivates the fact that $\Lip_H(W)$ is not a vector
space.

\begin{theorem}[Area formula]\label{th:area_formula}
Let $W\subset\mathbb{W}$ be an open set and let $\phi\colon W\to\R$ be a locally intrinsic Lipschitz function. Then the intrinsic epigraph $E_\phi\subset\He^n$ has locally finite $H$-perimeter in the cylinder 
\begin{equation*}
W*\R=\big\{ w*s\mathrm{e}_1\in\He^n : w\in W,\ s\in\R\big\},
\end{equation*}
and for $\leb^{2n}$-a.e. $w\in W$ the inner horizontal normal to $\de E_\phi$ is given by
\begin{equation}\label{eq:nu_epigraph}
\nu_{E_\phi}(\Phi(w))=\left(\frac{1}{\sqrt{1+|\nabla^\phi\phi(w)|^2}},\frac{-\nabla^\phi\phi(w)}{\sqrt{1+|\nabla^\phi\phi(w)|^2}}\right).
\end{equation}
Moreover, for any $W'\subset\subset W$, the following area formula holds:
\begin{equation}\label{eq:area_formula}
P_H(E_\phi;W'*\R)=\int_{W'} \sqrt{1+|\nabla^\phi\phi(w)|^2}\ d\leb^{2n}.
\end{equation}
\end{theorem}

Formula~\eqref{eq:nu_epigraph} for the inner horizontal normal to $\de E_\phi$ and the area formula~\eqref{eq:area_formula} are proved in~\cite{cittietal14}, respectively in Corollary~4.2 and in Theorem~1.6. The area formula~\eqref{eq:area_formula} can be improved in the following way
\begin{equation}\label{eq:general_area_formula}
\int_{\de E_\phi\cap W'*\R}g(p)\ d \mu_{E_\phi} =\int_{W'}g(\Phi(w))\sqrt{1+|\nabla^\phi\phi(w)|^2}\ d\leb^{2n},
\end{equation} 
where $g\colon\de E_\phi\to\R$ is a Borel function. 
To avoid long equations, in the following we   often omit the variables and the
flow map $\Phi$ when we   apply the area formula~\eqref{eq:area_formula} and
its general version~\eqref{eq:general_area_formula}.

\section{Intrinsic Lipschitz approximation}\label{sec:2}

In this section, we prove the following result, which contains
Theorem~\ref{th_intro:lip_app} in the Introduction as a particular case. 

\begin{theorem}\label{th:lip_app_easy}
Let $n\geq 2$. There exist positive dimensional constants $C_1(n)$, $\eps_1(n)$ and $\delta_1(n)$ with the following property. If $E\subset\He^n$ is a $(\Lambda,r_0)$-minimizer of $H$-perimeter in $C_{5124}$ with $\mathbf{e}(5124)\le\eps_1(n)$, $ 
\Lambda r_0\le1$, $r_0>5124$, and $ 0\in \partial E$, then, letting 
\begin{equation*}
M=C_1\cap\de E, \qquad M_0=\big\{q\in M : \sup_{0<s<256}\mathbf{e}(q,s)\le\delta_1(n)\big\}, 
\end{equation*}
  there exists an intrinsic Lipschitz function $\phi\colon\W\to\R$ such that
\begin{equation}\label{eq:lip_app_easy_1}
\sup_\W|\phi|\le C_1(n)\,\mathbf{e}(5124)^{\frac{1}{2(2n+1)}},\qquad \Lip_H(\phi)\le1, 
\end{equation}
\begin{equation}\label{eq:lip_app_easy_2}
M_0\subset M\cap\Gamma, \qquad \Gamma=\gr(\phi|_{D_1}),
\end{equation}
\begin{equation}\label{eq:lip_app_easy_3}
\haush^{2n+1}(M\bigtriangleup\Gamma)\le C_1(n)\,\mathbf{e}(5124),
\end{equation}
\begin{equation}\label{eq:lip_app_easy_4}
\int_{D_1}|\nabla^\phi\phi|^2\ d\leb^{2n}\le C_1(n)\,\mathbf{e}(5124).
\end{equation}
\end{theorem}

\begin{proof}
The proof is divided in three steps.

\medskip

\textit{Step~1: construction of $\phi$}. Let $\eps_0(n)$ and $C_0(n)$ be the constants given in Theorem~\ref{th:H-height_bound}. Then we have 
\begin{equation}\label{eq:lip_app_hgt_bound}
\sup\big\{|\hgt(p)| : p\in C_1\cap\de E\big\}\le C_0(n)\,\mathbf{e}(16)^{\frac{1}{2(2n+1)}},
\end{equation}
provided that $\mathbf{e}(16)\le\eps_0(n)$; this follows from the elementary properties of the excess with  $\eps_1(n)\le\eps_0(n)$ suitably small.

Let $q\in M_0$ and $p\in M$ be fixed. Then $p,q\in C_1$, so $d_C(p,q)<8$ by~\eqref{eq:D_and_C_comparison}, where $d_C$ is the quasi-distance induced by the quasi norm $\|\cdot\|_C$ defined in~\eqref{eq:C_quasi_norm}. We consider the blow-up of $E$ at scale $d_C(p,q)$ centered in $q$, that is, $F=E_{q,d_C(p,q)}=\delta_{1/r}(\tau_{q^{-1}} E)$ with $r=d_C(p,q)$. By Remark~\ref{remark:scaling_perim_min}, $F$ is a $(\Lambda',r_0')$-perimeter minimizer in $(C_{5124})_{q,d_C(p,q)}$, with
\begin{equation*}
\Lambda'=\Lambda\, d_C(p,q), \qquad r_0'=\frac{r_0}{d_C(p,q)}>1. 
\end{equation*}
Since 
\begin{equation*}
C_{16}\subset (C_{5124})_{q,d_C(p,q)}, \qquad \Lambda'r_0'\le1, \qquad 0\in\de F 
\end{equation*}
and, by the scaling property of the excess and by definition of $M_0$,
\begin{equation*}
\mathbf{e}(F,0,16,\nu)=\mathbf{e}(E,q,16d_C(p,q),\nu)\le\delta_1(n),
\end{equation*}
then, provided that $\delta_1(n)\le\eps_0(n)$, by Theorem~\ref{th:H-height_bound} we have 
\begin{equation*}
\sup\big\{|\hgt(w)| : w\in C_1\cap\de F\big\}\le C_0(n)\,\delta_1(n)^{\frac{1}{2(2n+1)}}.
\end{equation*} 
In particular, choosing 
\begin{equation*}
w=\frac{1}{d_C(p,q)}\,q^{-1}*p\in C_1\cap\de F,
\end{equation*}
we get
\begin{equation}\label{eq:estimate_lip_q}
|\hgt(q^{-1}*p)|\le C_0(n)\,\delta_1(n)^{\frac{1}{2(2n+1)}}d_C(p,q).
\end{equation}
We now set 
\begin{equation}\label{eq:def_L(n)}
L(n):=C_0(n)\,\delta_1(n)^{\frac{1}{2(2n+1)}}
\end{equation}
and we choose $\delta_1(n)$ so small that $L(n)<1$. Then, by~\eqref{eq:estimate_lip_q}, we conclude that $d_C(p,q)=\|\pi(q^{-1}*p)\|_\infty$ and we get
\begin{equation}\label{eq:estimate_lip_qp}
|\hgt(q^{-1}*p)|\le L(n)\| \pi(q^{-1}*p)\|_\infty \quad \text{for all } p\in M,\, q\in M_0.
\end{equation} 
In particular, \eqref{eq:estimate_lip_qp} proves that the projection $\pi$ is invertible on $M_0$. Therefore, we can define a function $\phi\colon\pi(M_0)\to\R$ setting $\phi(\pi(p))=\hgt(p)$ for all $p\in M_0$. From~\eqref{eq:estimate_lip_qp}, we deduce that
\begin{equation*}
|\phi(\pi(p))-\phi(\pi(q))|\le L(n)\|\pi(q^{-1}*p)\|_\infty \quad \text{for all } p,q\in M_0,
\end{equation*}
so that $\phi$ is an intrinsic Lipschitz function on $\pi(M_0)$ with $\Lip_H(\phi,\pi(M_0))\le L(n)<1$ by~\eqref{eq:def_intr_lip}. Since $M_0\subset M$, by~\eqref{eq:lip_app_hgt_bound} we also have
\begin{equation*}
|\phi(\pi(p))|\le C_0(n)\,\mathbf{e}(16)^{\frac{1}{2(2n+1)}} \quad \text{for all } p\in M_0.
\end{equation*} 
Therefore, by Proposition~\ref{prop:intrinsic_lip_ext}, possibly choosing
$\delta_1(n)$ smaller accordingly to~\eqref{eq:lip_const_extension}, we can
extend $\phi$ from $\pi(M_0)$  to the whole $\W$ with $\Lip_{H}(\phi,\W)\le
L(n)<1$ in such a way that
\begin{equation*}
M_0\subset M\cap\Gamma, \quad \Gamma=\gr(\phi|_{D_1}), \quad \text{and} \quad
|\phi(w)|\le C_0(n)\,\mathbf{e}(16)^{\frac{1}{2(2n+1)}} \quad \text{for all }
w\in\W.
\end{equation*}
We thus proved~\eqref{eq:lip_app_easy_1} and~\eqref{eq:lip_app_easy_2} for a suitable $C_1(n)\ge C_0(n)$.

\medskip

\textit{Step~2: covering argument}. We now prove~\eqref{eq:lip_app_easy_3} via a covering argument. By definition of $M_0$, for every $q\in M\setminus M_0$ there exists $s=s(q)\in(0,256)$ such that
\begin{equation}\label{eq:lip_app_besic}
\int_{C_s(q)\cap\de E}\frac{|\nu_E-\nu|_g^2}{2}\ d\haush^{2n+1}>\frac{\delta_1(n)}{\delta(n)}\,s^{2n+1},
\end{equation}  
 with $\delta(n)=\tfrac{2\omega_{2n-1}}{\omega_{2n+1}}$ and $\nu=-X_1$ as in
Definition~\ref{def:excess}. The family of balls
\begin{equation*}
\big\{B_{2s}(q): q\in M\setminus M_0,\ s=s(q)\big\}
\end{equation*}
is a covering of $M\setminus M_0$. By the $5r$-covering Lemma, there exist a sequence of points $q_h\in M\setminus M_0$ and a sequence of radii $s_h=s(q_h)$, $h\in\N$, with $q_h$ and $s_h$ satisfying~\eqref{eq:lip_app_besic}, such that the balls $B_{2s_h}(q_h)$ are pairwise disjoint and
\begin{equation*}
\big\{B_{10s_h}(q_h): h\in\N\big\}
\end{equation*} 
is still a covering of $M\setminus M_0$. Note that $B_{10s_h}(q_h)\subset
C_{5124}$, because if $p\in B_{10s_h}(q_h)$ then,
by~\eqref{eq:quasi_C_equiv_infty},
\begin{equation*}
\|p\|_C\le2\|p\|_\infty\le 2d_\infty(p,q_h)+2\|q_h\|_\infty<20s_h+4\|q_h\|_C<5124.
\end{equation*}
Therefore, by the density estimates in~\cite{monti&vittone15}*{Theorem~4.1}, we
get
\begin{align*}
\haush^{2n+1}(M\setminus M_0)&\le\sum_{h\in\N}\haush^{2n+1}\big((M\setminus M_0)\cap B_{10s_h}(q_h)\big)\\
	&\le\sum_{h\in\N}\haush^{2n+1}\big(M\cap B_{10s_h}(q_h)\big)\\
	&\le C(n)\sum_{h\in\N} s_h^{2n+1},
\end{align*}
where  $C(n)$ is a positive dimensional constant. Since $C_{s_h}(q_h)\subset B_{2s_h}(q_h)$ by~\eqref{eq:D_and_C_comparison}, the cylinders $C_{s_h}(q_h)$ are pairwise disjoint and contained in $C_{5124}$, so we have
\begin{equation}\label{eq:lip_app_besic2}
\haush^{2n+1}(M\setminus M_0)\le C(n)\sum_{h\in\N}\int_{C_{s_h}(q_h)\cap\de E}\frac{|\nu_E-\nu|_g^2}{2}\ d\haush^{2n+1}\le C(n)\,\mathbf{e}(5124), 
\end{equation}
where  $C(n)$ is a new positive dimensional constant. Therefore, since $M\setminus\Gamma\subset M\setminus M_0$, by~\eqref{eq:lip_app_besic2} it follows that
\begin{equation}\label{eq:1_half}
\haush^{2n+1}(M\setminus\Gamma)\le C(n)\,\mathbf{e}(5124),
\end{equation}
which is the first half of~\eqref{eq:lip_app_easy_3}.

We now bound the second half of~\eqref{eq:lip_app_easy_3}. We choose $\eps_1(n)$ so small that
\begin{equation*}
\mathbf{e}(2)\le\omega(n,\tfrac{1}{2},\tfrac{1}{5124},5124),
\end{equation*} 
where $\omega(n,t,\Lambda,r_0)$, with $t\in(0,1)$, is the
constant given in~\cite{monti&vittone15}*{Lemma~3.3}. 
This is possible by the scaling property of the excess. Then, by~(3.57)
in~\cite{monti&vittone15}*{Lemma~3.4}, we have
\begin{equation*}
\leb^{2n}(G)\le\haush^{2n+1}\left(M\cap\pi^{-1}(G)\right)
\end{equation*} 
for any Borel set $G\subset D_1$. Therefore, by the area formula~\eqref{eq:area_formula} in Theorem~\ref{th:area_formula}, we can estimate
\begin{align}\label{eq:2_half_1}
\delta(n)\,\haush^{2n+1}(\Gamma\setminus M)&=\int_{\pi(\Gamma\setminus M)}\sqrt{1+|\nabla^\phi\phi(w)|^2}\ d\leb^{2n}\nonumber\\
	&\le\sqrt{1+\|\nabla^\phi\phi\|_{L^\infty(D_1)}^2}\,\leb^{2n}\big(\pi(\Gamma\setminus M)\big)\nonumber\\
	&\le \sqrt{1+\|\nabla^\phi\phi\|_{L^\infty(D_1)}^2}\,\haush^{2n+1}\left(M\cap\pi^{-1}\big(\pi(\Gamma\setminus M)\big)\right).  
\end{align}
Since $\phi$ is intrinsic Lipschitz on $D_1$ with $\Lip_H(\phi)<1$ by construction, by~\cite{cittietal14}*{Proposition~4.4} there exists a positive dimensional constant $C(n)$ such that
\begin{equation}\label{eq:2_half_2}
\|\nabla^\phi\phi\| _{L^\infty(D_1)}\le C(n)\Lip_H(\phi)\big(\Lip_H(\phi)+1\big)<2C(n).
\end{equation} 
Thus, by~\eqref{eq:2_half_1} and~\eqref{eq:2_half_2}, there exists a positive dimensional constant $C(n)$ such that
\begin{equation}\label{eq:2_half_3}
\haush^{2n+1}(\Gamma\setminus M)\le C(n)\,\haush^{2n+1}\left(M\cap\pi^{-1}\big(\pi(\Gamma\setminus M)\big)\right).
\end{equation}
Since we have
\begin{equation*}
M\cap\pi^{-1}\big(\pi(\Gamma\setminus M)\big)\subset M\setminus\Gamma,
\end{equation*}
by~\eqref{eq:1_half} and~\eqref{eq:2_half_3} we conclude that, for some positive dimensional constant $C'(n)$,
\begin{equation}\label{eq:2_half}
\haush^{2n+1}(\Gamma\setminus M)\le C(n)\,\haush^{2n+1}(M\setminus\Gamma)\le C'(n)\,\mathbf{e}(5124),
\end{equation}
which is the second half of~\eqref{eq:lip_app_easy_3}. Combining~\eqref{eq:1_half} and~\eqref{eq:2_half}, we prove~\eqref{eq:lip_app_easy_3}.

\medskip

\textit{Step~3: $L^2$-estimate}. Finally, we prove~\eqref{eq:lip_app_easy_4}. We first notice that, by Theorem~\ref{th:area_formula} and~\cite{ambrosio&scienza10}*{Corollary~2.6}, for $\haush^{2n+1}$-a.e. $p\in M\cap\Gamma$ there exists $\lambda(p)\in\{-1,1\}$ such that 
\begin{equation}\label{eq:L2_1}
\nu_{E}(p)=\lambda(p)\frac{\big(1,-\nabla^\phi\phi(\pi(p))\big)}{\sqrt{1+|\nabla^\phi\phi(\pi(p))|^2}}.
\end{equation}
Taking into account that, for $\haush^{2n+1}$-a.e. $p\in M\cap\Gamma$,
\begin{equation}\label{eq:L2_1.5}
\frac{|\nu_E(p)-\nu(p)|_g^2}{2}=1-\langle  \nu_E(p),\nu(p)\rangle _g\ge\frac{1-\langle \nu_E(p),\nu(p)\rangle _g^2}{2},
\end{equation}
by~\eqref{eq:L2_1} and by the  area formula~\eqref{eq:general_area_formula} we find that
\begin{align*}
\mathbf{e}(1)&\ge\int_{M\cap\Gamma}\frac{1-\langle \nu_E(p),\nu(p)\rangle _g^2}{2}\ d{\mu_E}\\
	&=\frac{1}{2}\int_{M\cap\Gamma}\frac{|\nabla^\phi\phi(\pi(p))|^2}{1+|\nabla^\phi\phi(\pi(p))|^2}\ d \mu_E \\
	&=\frac{1}{2}\int_{\pi(M\cap\Gamma)}\frac{|\nabla^\phi\phi(w)|^2}{\sqrt{1+|\nabla^\phi\phi(w)|^2}}\ d\leb^{2n}.
\end{align*}
Recalling~\eqref{eq:2_half_2} and the scaling property of the excess, we conclude that there exists a positive dimensional constant $C(n)$ such that
\begin{equation}\label{eq:L2_2}
\int_{\pi(M\cap\Gamma)}|\nabla^\phi\phi(w)|^2\ dw\le C(n)\,\mathbf{e}(5124).
\end{equation}
Moreover, again by the  area formula~\eqref{eq:general_area_formula}, there exists a positive dimensional constant $C(n)$ such that
\begin{align*}
\int_{\pi(M\bigtriangleup\Gamma)}|\nabla^\phi\phi(w)|^2\ d\leb^{2n} &=\int_{M\bigtriangleup\Gamma}\frac{|\nabla^\phi\phi(\pi(p))|^2}{\sqrt{1+|\nabla^\phi\phi(\pi(p))|^2}}\ d \mu_E\\
	&\le C(n)\|\nabla^\phi\phi\|_{L^\infty(D_1)}^2\,\haush^{2n+1}(M\bigtriangleup\Gamma).
\end{align*}
By~\eqref{eq:2_half_2} and~\eqref{eq:lip_app_easy_3}, we find a positive dimensional constant $C(n)$ such that
\begin{equation}\label{eq:L2_3}
\int_{\pi(M\bigtriangleup\Gamma)}\|\nabla^\phi\phi(w)|^2\ dw\le C(n)\,\mathbf{e}(5124).
\end{equation}
Combining~\eqref{eq:L2_2} and~\eqref{eq:L2_3}, we prove~\eqref{eq:lip_app_easy_4}.
\end{proof}

\begin{remark}[$\sigma$-representative]\label{remark:sigma_repres}
Let $0<\sigma\le 1$ and $I=(-1,1)$. We let $\mathcal{A}(\sigma)$ be the family of sets $A\subseteq D_\sigma$ such that
\begin{equation*}
|\hgt(q^{-1}*p)|\le L(n)\|\pi(q^{-1}*p)\|_\infty \quad \text{for all } p\in M\cap D_\sigma*I,\ q\in M\cap A*I,
\end{equation*}
where $L(n)$ is the dimensional constant  in~\eqref{eq:def_L(n)}. The family
$\mathcal{A}(\sigma)$ is partially ordered by inclusion and is closed under
union. Thus $\mathcal{A}(\sigma)$ has a unique maximal element $A^\star_\sigma$.
Then, by~\eqref{eq:estimate_lip_qp}, we have that
\begin{equation*}
|\hgt(q^{-1}*p)|\le L(n)\|\pi(q^{-1}*p)\|_\infty \quad \text{for all } p,q\in M_0\cup (M\cap A^\star_\sigma*I).
\end{equation*}
Therefore, in Step~1 of the proof of Theorem~\ref{th:lip_app_easy}, it is not restrictive to assume that the intrinsic Lipschitz approximation $\phi\colon\W\to\R$ is defined in such a way that
\begin{equation*}
\phi(\pi(p))=\hgt(p) \quad \text{for all } p\in M_0\cup (M\cap A^\star_\sigma*I).
\end{equation*} 
We define such an intrinsic Lipschitz function a \emph{$\sigma$-representative} of Theorem~\ref{th:lip_app_easy}. 
 Moreover, if Theorem~\ref{th:lip_app_easy} is applied with a scaling factor
$\lambda>0$, then we have $0<\sigma\le\lambda$, $I=(-\lambda,\lambda)$ and we
can define in the same way the family $\mathcal{A}(\sigma,\lambda)$, its maximal
element $A^\star_{\sigma,\lambda}$ and a
\emph{$(\sigma,\lambda)$-representative} of Theorem~\ref{th:lip_app_easy}. 
\end{remark}

\section{Local maximal functions}

In this section, we prove some lemmas on maximal functions of measures that are used in the proof of Theorem \ref{th_intro:alpha_app}.

\subsection{Maximal function on disks}

Given $s>0$ and a non-negative measure $\mu$ on $D_{4s} \subset\W$, the \emph{local maximal function} of $\mu$ is defined as
\begin{equation}\label{eq:def_loc_max_func}
M\mu(x):=\sup_{0<r<4s-\|x\|_\infty}\dfrac{\mu(D_r(x))}{\kappa_n r^{2n+1}} \qquad\text{for } x\in D_{4s},
\end{equation}
where $\kappa_n=\leb^{2n}(D_1)$ as in~\eqref{eq:leb_meas_disk}.

\begin{lemma}\label{lemma:max_func_disks}
Let $s>0$ and let $\mu\colon D_{4s}\to[0,\infty)$ be as above. Assume that $\theta>0$ is such that 
\begin{equation}\label{eq:est_theta}
\mu(D_{4s})\le\dfrac{\theta}{5^{2n+1}}\kappa_n s^{2n+1}
\end{equation}
and define
\begin{equation*}
J_\theta=\big\{x\in D_{4s}: M\mu(x)>\theta\big\}.
\end{equation*}
Then for all $r\leq 3s$ we have
\begin{equation}\label{eq:meas_J_theta}
\leb^{2n}(J_\theta\cap D_r)\leq\dfrac{5^{2n+1}}{\theta}\mu(J_{\theta/2^{2n+1}}\cap D_{r+\frac{s}{5}}).
\end{equation}
\end{lemma}

\begin{proof}
Let $r\leq3s$ be fixed. If $x\in J_\theta\cap D_r$, then there exists $r_x>0$ such that 
\begin{equation*}
\mu(D_{r_x}(x))>\theta\kappa_n r_x^{2n+1}.
\end{equation*}
By the 5r-covering Lemma applied to the family $\{D_{r_x}(x): x\in J_\theta\cap D_r\}$, we find a sequence of pairwise disjoint balls $\{D_{r_i}(x_i)\}_{i\in\N}$, with $x_i\in J_\theta\cap D_r$ and $r_i>0$, such that 
\begin{equation*}
J_\theta\cap D_r\subset\bigcup_{x\in J_\theta\cap D_r} D_{r_x}(x)\subset\bigcup_{i\in\N}D_{5r_i}(x_i), 
\qquad
\mu(D_{r_i}(x_i))>\theta\kappa_n r_i^{2n+1}.
\end{equation*}
In particular, by~\eqref{eq:est_theta}, we get
\begin{equation*}
r_i<\sqrt[2n+1]{\frac{\mu(D_{r_i}(x_i))}{\theta\kappa_n}}\leq\sqrt[2n+1]{\frac{\mu(D_{4s})}{\theta\kappa_n}}\leq\dfrac{s}{5},
\end{equation*}
and so, for any $i\in\N$, we have 
\begin{equation*}
D_{r_i}(x_i)\subset D_{\|x_i\|_\infty+r_i}\subset D_{r+\frac{s}{5}}.
\end{equation*}
We claim that
\begin{equation*}
D_{r_i}(x_i)\subset J_{\theta/2^{2n+1}}\cap D_{r+\frac{s}{5}}
\end{equation*}
for any $i\in\N$. Indeed,  by contradiction  assume that 
there exists  $y\in D_{r_i}(x_i)$  such that
$M\mu(y)\leq\frac{\theta}{2^{2n+1}}$. Then $D_{r_i}(x_i)\subset D_{2r_i}(y)$ and
\begin{equation*}
4s-\|y\|_\infty\geq 4s-r-\frac{s}{5}\geq 4s-3s-\frac{s}{5}=\frac{4}{5}s>2r_i.
\end{equation*}
Hence, we have
\begin{align*}
\frac{\theta}{2^{2n+1}}\geq M\mu(y)&=\sup_{0<\delta<4s-\|y\|_\infty}\frac{\mu(D_\delta(y))}{\kappa_n \delta^{2n+1}}\\
&\ge\sup_{2r_i<\delta<4s-\|y\|_\infty}\frac{\mu(D_\delta(y))}{\kappa_n \delta^{2n+1}}\\
&\geq\sup_{2r_i<\delta<4s-\|y\|_\infty}\frac{\mu(D_{r_i}(x_i))}{\kappa_n \delta^{2n+1}}=\frac{\mu(D_{r_i}(x_i))}{\kappa_n (2r_i)^{2n+1}}>\frac{\theta}{2^{2n+1}}, 
\end{align*}
a contradiction.

We can finally estimate:
\begin{align*}
\leb^{2n}(J_\theta\cap D_r)&\leq\sum_{i\in\N}\leb^{2n}(D_{5r_i}(x_i))=5^{2n+1}\,\kappa_n \sum_{i\in\N}r_i^{2n+1}\leq \dfrac{5^{2n+1}}{\theta}\sum_{i\in\N}\mu(D_{r_i}(x_i))\\
&=\dfrac{5^{2n+1}}{\theta}\mu\left(\bigcup_{i\in\N}D_{r_i}(x_i)\right)\leq\dfrac{5^{2n+1}}{\theta}\mu(J_{\theta/2^{2n+1}}\cap D_{r+\frac{s}{5}}), 
\end{align*}
and~\eqref{eq:meas_J_theta} follows.
\end{proof}

\subsection{Maximal function on \texorpdfstring{$\phi$}{phi}-balls}

We recall the  Poincaré inequality 
for intrinsic Lipschitz functions.
 The notion of intrinsic Lipschitz function can be equivalently restated on bounded open sets introducing a suitable notion of graph distance, see~\cite{cittietal14}*{Definition~1.1} or~\cite{cittietal16}.  
 Let $W\subset\mathbb{W}$ be set and let $\phi\colon W\to\R$ be a function. The map $d_\phi\colon W\times W\to[0,\infty)$ given by
\begin{equation}\label{eq:def_graph_dist}
d_\phi(w,w')=\frac{1}{2}\bigg(\big\| \pi\big(\Phi(w)^{-1}*\Phi(w')\big)\big\|_\infty+\big\|\pi\big(\Phi(w')^{-1}*\Phi(w)\big)\big\|_\infty\bigg)
\end{equation}
for any $w,w'\in W$, where $\Phi(w)=w*\phi(w)\mathrm{e}_1$ for all $w\in W$, is the \emph{graph distance} induced by $\phi$.

Comparing~\eqref{eq:def_intr_lip} with~\eqref{eq:def_graph_dist}, it is easy to
see that, 
if $W\subset\mathbb{W}$ is a bounded open set and $\phi\colon W\to\R$ is a
continuous function, then $\phi$ is an intrinsic $L$-intrinsic Lipschitz
function if and only if
\begin{equation*}
|\phi(w)-\phi(w')|\le Ld_\phi(w,w') ,\quad   w,w'\in W.
\end{equation*}    

If $\phi$ is an intrinsic $L$-Lipschitz function on $W$, then $d_\phi$ turns out to be a quasi-distance on $W$, that is, $d_\phi(x,y)=0$ if and only if $x=y$ for all $x,y\in W$, $d_\phi$ is symmetric and, for all $x,y,z\in W$, 
\begin{equation}\label{eq:quasi_triang_ineq}
d_\phi(x,y)\le c_L(d_\phi(x,z)+d_\phi(z,y)),
\end{equation} 
where $c_L\ge1$ depends only on $L$ and 
\begin{equation}\label{eq:c_L_to_1}
\lim_{L\to0}c_L=1,
\end{equation}
see~\cite{cittietal14}*{Section~3}.

The following 
Poincaré inequality is proved in~\cite{cittietal16},  see Theorem~1.2 and also
Corollary~1.3 therein for the case $p=1$.

\begin{theorem}[Poincaré inequality]\label{th:poincare_ineq}
Let $W\subset\W\subset \mathbb H^n$, $n\geq 2$,  be a bounded open set  and let
$1\le p<\infty$.  Then there exist two constants $C_1^L,C_2^L>0$ with $C_2^L>1$,
depending on $L>0$, such that  for any $L$-intrinsic Lipschitz function
$\phi\colon W\to\R$ we have  
\begin{equation}
\label{eq:poincare_ineq_intrinsic_lip}
\int_{U_\phi(x,r)}|\phi-(\phi)_{x,r}|^p\  d\leb^{2n}\le C_1^L
r^p\int_{U_\phi(x,C_2^L r)}|\nabla^\phi\phi|^p\ d\leb^{2n}
\end{equation} 
for every $U_\phi(x,C_2^L r)\subset W$, where 
\begin{equation}\label{eq:intrinsic_ball}
U_\phi(x,r)=\{y\in W : d_\phi(x,y)<r\}
\end{equation} 
and
\begin{equation*}
(\phi)_{x,r}=\aint_{U_\phi(x,r)}\phi\ d\leb^{2n}
=\frac{1}{\leb^{2n}(U_\phi(x,r))}\int_{U_\phi(x,r)}\phi\ d\leb^{2n}.
\end{equation*}
\end{theorem}

\noindent 
For future convenience, we define 
\begin{equation}\label{eq:gamma_2}
\gamma_2(n)=\lim_{L\to0}C_2^L\ge1.
\end{equation}  

The $\leb^{2n}$-measure of the ball $U_\phi(x,r)$ defined in~\eqref{eq:intrinsic_ball} is comparable to $r^{2n+1}$. Namely, there exist two constants $c_1^L,c_2^L>0$ depending on $L$ such that, for all $U_\phi(x,r)\subset W$, we have
\begin{equation}\label{eq:meas_phi_balls}
c_1^L\le\frac{\leb^{2n}(U_\phi(x,r))}{r^{2n+1}}\le c_2^L,
\end{equation}
see~\cite{cittietal16}*{Section~2.3} and the references therein.

We can now introduce the \emph{local $\phi$-maximal function}. Let $n\ge2$, $s>0$, and let $\phi\colon\W\to\R$ be an $L$-intrinsic Lipschitz function. By~\eqref{eq:c_L_to_1} and by~\eqref{eq:gamma_2}, there exists a dimensional constant $\ell(n)>0$ such that 
\begin{equation}\label{eq:def_ell}
L\in[0,\ell(n)] \implies c_L\le2 \text{ and } C_2^L\le2\gamma_2(n),
\end{equation}
where $c_L$ is as in~\eqref{eq:quasi_triang_ineq} and $C_2^L$ is as in Theorem~\ref{th:poincare_ineq}. For all $L\in[0,\ell(n)]$, we define the \emph{local $\phi$-maximal function} of $\mu_\phi$ as
\begin{equation}\label{eq:def_loc_phi_max_func}
[\mu_\phi](x):=\sup_{0<r<r_\phi(x,s)}\dfrac{\mu_\phi(U_\phi(x,r))}
{\leb^{2n}(U_\phi(x,r))}, \qquad x\in U_\phi(0,s),
\end{equation}
where we set 
\begin{equation}\label{eq:def_r_phi(x,s)}
r_\phi(x,s)=\frac{\rho(n)}{c_L}\,s-d_\phi(x,0), \qquad x\in U_\phi(0,s),
\end{equation}
the dimensional constant is
\begin{equation}\label{eq:def_rho}
\rho(n)=64\gamma_2(n)+2,
\end{equation}
and the non-negative measure $\mu_\phi$ on $U_\phi(0,\rho(n)s)$ is given by 
\begin{equation*}
d\mu_\phi=|\nabla^\phi\phi|\,d\leb^{2n}.
\end{equation*}
The maximal function introduced in~\eqref{eq:def_loc_phi_max_func} is well-defined, since
\begin{equation*}
x\in U_\phi(0,s),\ r<r_\phi(x,s) \implies U_\phi(x,r)\subset U_\phi(0,\rho(n)s), 
\end{equation*}
by the quasi-triangular inequality~\eqref{eq:quasi_triang_ineq}. 

We use the Poincaré inequality~\eqref{eq:poincare_ineq_intrinsic_lip} to prove the following result on $[\mu_\phi]$.

\begin{lemma}\label{lemma:phi_max_func}
Let $n\ge2$, $s>0$, $\phi\colon\W\to\R$, $\mu_\phi$, $[\mu_\phi]$, $L\in[0,\ell(n)]$ be as above. Let $\theta>0$ and define
\begin{equation}\label{eq:def_J_theta_phi}
J^\phi_\theta=\big\{x\in U_\phi(0,s): [\mu_\phi](x)>\theta\big\}.
\end{equation}
Then there exists a constant $C=C(n,L)$ such that for all $x,y\in U_\phi(0,s)\setminus J^\phi_\theta$ we have
\begin{equation}\label{eq:phi_theta_lip}
|\phi(x)-\phi(y)|\le C\theta\, d_\phi(x,y).
\end{equation}
\end{lemma}

\begin{proof}
Let $x\in U_\phi(0,s)\setminus J^\phi_\theta$ and let $C_2^L r<r_\phi(x,s)$. Then, by Theorem~\ref{th:poincare_ineq} with $p=1$, we have
\begin{equation*}
\int_{U_\phi(x,r)}|\phi-(\phi)_{x,r}|\ d\leb^{2n}\le C_1^L r\int_{U_\phi(x,C_2^L r)}|\nabla^\phi\phi|\ d\leb^{2n}=C_1^L r\,\mu_\phi(U_\phi(x,C_2^L r)).
\end{equation*} 
By~\eqref{eq:def_loc_phi_max_func} and by~\eqref{eq:def_J_theta_phi}, we have
\begin{equation*}
\mu_\phi(U_\phi(x,C_2^L r))\le\theta\,\leb^{2n}(U_\phi(x,C_2^L r)).
\end{equation*}
Therefore, by~\eqref{eq:meas_phi_balls}, we have
\begin{equation*}
\int_{U_\phi(x,r)}|\phi-(\phi)_{x,r}|\ d\leb^{2n}\le C_1^L r \theta c_2^L(C_2^L r)^{2n+1}, 
\end{equation*} 
and so, again by~\eqref{eq:meas_phi_balls}, we get
\begin{equation*}
\aint_{U_\phi(x,r)}|\phi-(\phi)_{x,r}|\ d\leb^{2n}\le\dfrac{c_2^L}{c_1^L}\,C_1^L (C_2^L)^{2n+1}\theta r, 
\end{equation*}
for all $x\in U_\phi(0,s)\setminus J^\phi_\theta$ and $C_2^L r<r_\phi(x,s)$.

In particular, for all $j=0,1,2,\dots$, we have
\begin{align*}
|(\phi)_{x,\frac{r}{2^{j+1}}}-(\phi)_{x,\frac{r}{2^j}}|&\leq\aint_{U_\phi(x,\frac{r}{2^{j+1}})}|\phi(u)-(\phi)_{x,\frac{r}{2^j}}|\ d\leb^{2n}(u) \\
	&\leq 2^{2n+1}\frac{c_2^L }{c_1^L }\aint_{U_\phi(x,\frac{r}{2^j})}|\phi(u)-(\phi)_{x,\frac{r}{2^j}}|\ d\leb^{2n}(u) \\
	&\leq\dfrac{2^{2n+1}}{2^j}\left(\dfrac{c_2^L}{c_1^L}\right)^2 C_1^L (C_2^L)^{2n+1}\theta r.	
\end{align*}
Since $\phi$ is continuous, we get
\begin{equation*}
|\phi(x)-(\phi)_{x,r}|\le\sum_{j=0}^{\infty}\left|(\phi)_{x,\frac{r}{2^{j+1}}}-(\phi)_{x,\frac{r}{2^j}}\right|\le 2^{2n+2}\left(\dfrac{c_2^L}{c_1^L}\right)^2 C_1^L (C_2^L)^{2n+1}\theta r,
\end{equation*}
for all $x\in U_\phi(0,s)\setminus J^\phi_\theta$ and $C_2^L r<r_\phi(x,s)$.

Finally, let $x,y\in U_\phi(0,s)\setminus J^\phi_\theta$, $r=d_\phi(x,y)$ and $c_3^L=2c_L$. Then, by the quasi-triangular inequality~\eqref{eq:quasi_triang_ineq}, we have
\begin{equation*}
U_\phi(x,r)\cup U_\phi(y,r)\subset U_\phi(x,c_3^L r)\cap U_\phi(y,c_3^L r).
\end{equation*}
Notice that, again by~\eqref{eq:quasi_triang_ineq}, we have
\begin{equation*}
x,y\in U_\phi(0,s),\ r=d_\phi(x,y) \implies U_\phi(x,c_3^L r)\cup U_\phi(y,c_3^L r)\subset U_\phi(0,\rho(n)s),
\end{equation*}
because, by~\eqref{eq:def_ell} and~\eqref{eq:def_rho},
\begin{equation*}
c_L(2c_Lc_3^L+1)=c_L(4c_L^2+1)\le\rho(n).
\end{equation*}
Therefore we obtain
\begin{align*}
|(\phi)_{x,c_3^L r}-(\phi)_{y,c_3^L r}|&\leq\aint_{U_\phi(x,c_3^L r)\cap U_\phi(y,c_3^L r)}|\phi(u)-(\phi)_{x,c_3^L r}|+|\phi(u)-(\phi)_{x,c_3^L r}|\ d\leb^{2n}(u) \\
	&\le\dfrac{c_2^L}{c_1^L}(c_3^L)^{2n+1}\bigg(\aint_{U_\phi(x,c_3^L r)}|\phi(u)-(\phi)_{x,c_3^L r}|\ d\leb^{2n}(u)+\\
	&\hspace*{4.5cm}+\aint_{U_\phi(y,c_3^L r)}|\phi(u)-(\phi)_{y,c_3^L r}|\ d\leb^{2n}(u)\bigg). 
\end{align*}
Since $x,y\in U_\phi(0,s)\setminus J^\phi_\theta$, by~\eqref{eq:def_loc_phi_max_func} and by~\eqref{eq:def_J_theta_phi} we have
\begin{equation*}
\mu_\phi(U_\phi(x,c_3^L C_2^L r))\le\theta\,\leb^{2n}(U_\phi(x,c_3^L C_2^L r))
\end{equation*}
and, analogously, 
\begin{equation*}
\mu_\phi(U_\phi(y,c_3^L C_2^L r))\le\theta\,\leb^{2n}(U_\phi(y,c_3^L C_2^L r)),
\end{equation*}
provided that
\begin{equation*}
c_3^L C_2^L d_\phi(x,y)<\min\{r_\phi(x,s),r_\phi(y,s)\}.
\end{equation*}
By~\eqref{eq:def_ell}, since $x,y\in U_\phi(0,s)$, we have
\begin{equation*}
\min\{r_\phi(x,s),r_\phi(y,s)\}>\frac{\rho(n)s}{c_L}-s\ge\left(\frac{\rho(n)}{2}-1\right)s
\end{equation*}
and
\begin{equation*}
c_3^L C_2^L d_\phi(x,y)< 4c_L^2C_2^L s\le32\gamma_2(n)s,
\end{equation*}
so it is enough to check that
\begin{equation*}
32\gamma_2(n)\le\frac{\rho(n)}{2}-1,
\end{equation*}
but this is true thanks to the definition of $\rho(n)$ in~\eqref{eq:def_rho}. 

We can now conclude the proof. Let $x,y\in U_\phi(0,s)\setminus J^\phi_\theta$
and $r=d_\phi(x,y)$.  Then we have  
\begin{align*}
|\phi(x)-\phi(y)|&\le|\phi(x)-(\phi)_{x,c_3^L r}|+|(\phi)_{x,c_3^L r}-(\phi)_{y,c_3^L r}|+|\phi(y)-(\phi)_{y,c_3^L r}| \\
	&\le\left(2(c_3^L)^{2n+2}+2^{2n+3}c_3^L\right)\left(\frac{c_2^L}{c_1^L}\right)^2 C_1^L (C_2^L)^{2n+1}\theta r \\
	&=C(n,L)\theta\, d_\phi(x,y)
\end{align*} 
and~\eqref{eq:phi_theta_lip} follows. 
\end{proof}

\section{Approximation via maximal functions}\label{sec:3}

In this section, we develop the ideas contained in~\cite{delellis&spadaro11-1}*{Appendix~A} to prove the following result. In the proof, we
use Theorem~\ref{th:lip_app_easy} with a suitable scaling factor.

\begin{theorem}\label{th:lip_alpha_easy}
Let $n\geq 2$ and $\alpha\in(0,\frac{1}{2})$. There exist positive constants $C_2(n)$, $\eps_2(\alpha,n)$ and $k_2=k_2(n)$ with the following property. For any set 
$E\subset\He^n$ that is a $(\Lambda,r_0)$-minimizer of $H$-perimeter in $C_{k_2}$ with $\mathbf{e}(k_2)\le\eps_2(\alpha,n)$, $\Lambda r_0\le1$, $r_0>k_2$ and $0\in \partial E$,
there exist a function $\phi\colon\W\to\R$ and a set $K\subset D_1$ such that 
\begin{equation}\label{eq:lip_alpha_easy_3}
\leb^{2n}(D_1\setminus K)\le C_2(n)\,\mathbf{e}(k_2)^{1-2\alpha},
\end{equation}
\begin{equation}\label{eq:lip_alpha_easy_1}
\gr(\phi|_K)=\de E\cap \big(K*(-1,1)\big),
\end{equation}
\begin{equation}\label{eq:lip_alpha_easy_2}
\Lip_H(\phi|_K)\le C_2(n)\,\mathbf{e}(k_2)^\alpha.
\end{equation}
\end{theorem}

We need some preliminaries. The following result is an easy consequence of Cauchy--Schwarz inequality.

\begin{lemma}\label{lemma:BV_est}
Let $W\subset\W$ be an open set and let $\phi\colon W\to\R$ be an $L$-intrinsic Lipschitz function. For any Borel set $A\subset\subset W$, we have
\begin{equation}\label{eq:BV_est}
\left(\int_A|\nabla^\phi\phi|\ d\leb^{2n}\right)^2\le\sqrt{1+\|\nabla^\phi\phi\|^2_{L^\infty(W)}}\,\leb^{2n}(A)\int_{\gr(\phi|_A)}\frac{|\nabla^\phi\phi|^2}{1+|\nabla^\phi\phi|^2}\ d \mu_{E_\phi}.
\end{equation}
\end{lemma}

\begin{proof}
Let $A\subset\subset W$ be fixed. Then, by the area formula~\eqref{eq:general_area_formula},
\begin{align*}
\int_A|\nabla^\phi\phi|\ d\leb^{2n}&=\int_{\gr(\phi|_A)}\frac{|\nabla^\phi\phi|}{\sqrt{1+|\nabla^\phi\phi|^2}}\ d\mu_{E_\phi}  \\
	&\le\left(\int_{\gr(\phi|_A)}d \mu_{E_\phi} \right)^{\frac{1}{2}}\left(\int_{\gr(\phi|_A)} \frac{|\nabla^\phi\phi|^2} {1+|\nabla^\phi\phi|^2}\ d \mu_{E_\phi} \right)^{\frac{1}{2}} \\
	&=\left(\int_A\sqrt{1+|\nabla^\phi\phi|^2}\ d\leb^{2n}\right)^{\frac{1}{2}}\left(\int_{\gr(\phi|_A)}\frac{|\nabla^\phi\phi|^2}{1+|\nabla^\phi\phi|^2}\ d\mu_{E_\phi}\right)^{\frac{1}{2}} 
\\		&\le\sqrt[4]{1+\|\nabla^\phi\phi\|^2_{L^\infty(W)}}\,\leb^{2n}(A)^{\frac{1}{2}}\left(\int_{\gr(\phi|_A)}\frac{|\nabla^\phi\phi|^2}{1+|\nabla^\phi\phi|^2}\ d \mu_{E_\phi} \right)^{\frac{1}{2}}
\end{align*}
and~\eqref{eq:BV_est} follows squaring both sides.
\end{proof}

The following lemma compares the distance $d_\phi$ with the distance of points of the graph of an intrinsic Lipschitz function $\phi$.

\begin{lemma}\label{lemma:comp_d_phi_graph}
Let $W\subset\W$ be an open set and let $\phi\colon W\to\R$ be an intrinsic Lipschitz function. Then, for all $x\in W$, $r>0$ and $0<C<1/(1+\Lip_H(\phi))$, we have
\begin{equation}\label{eq:comp_d_phi_graph}
U_\phi(x,Cr)\subset\pi\big(B_r(\Phi(x))\cap\gr(\phi)\big)\subset U_\phi(x,r),
\end{equation}
where $U_\phi(x,r)$ is as in~\eqref{eq:intrinsic_ball} and $\Phi(x)=x*\phi(x)\mathrm{e}_1$.
\end{lemma}

\noindent For the proof, see~\cite{cittietal14}*{Proposition~3.6}.

Finally, the following result compares the distance $d_\phi$ with the distance $d_\infty$ in $W$. Its proof easily follows from the definition of $d_\phi$ in~\eqref{eq:def_graph_dist} and is left to the reader.

\begin{lemma}\label{lemma:U_phi_comp_disk}
Let $W\subset\W$ be an open set and let $\phi\colon W\to\R$ be a bounded intrinsic Lipschitz function. Then, for all $x\in W$, and $r>0$, we have
\begin{equation*}
U_\phi(x,r)\subset D_R(x) \qquad\text{and}\qquad D_r(x)\subset U_\phi(x,R),
\end{equation*}
where $R=r+2\|\phi\|_{L^\infty(W)}^{1/2}r^{1/2}$.
\end{lemma} 

\begin{proof}[Proof of Theorem~\ref{th:lip_alpha_easy}]
The proof is divided in three steps.

\medskip

\textit{Step~1: construction of $\phi$, $K$ and proof of~\eqref{eq:lip_alpha_easy_1}}. 
Let $\alpha\in(0,\frac{1}{2})$ be fixed. We assume $\eps_2(n,\alpha)\le\eps_1(n)$ and $k_2>5124$. 
Apply Theorem~\ref{th:lip_app_easy} with scaling factor~$\frac{k_2}{5124}$ and let $\phi\colon\W\to\R$ be the corresponding approximating function. Without
loss of generality, we can assume that $\phi$ is a
$(1,\frac{k_2}{5124})$-representative in the sense of
Remark~\ref{remark:sigma_repres}. Moreover, choosing $\eps_2(n,\alpha)$
sufficiently small, we can also assume that $\sup_\W|\phi|<1$. 
 
Let $I=(-\frac{k_2}{5124},\frac{k_2}{5124})$  and let $A\subset
D_{\frac{k_2}{5124}}$ be a Borel set. By~\eqref{eq:L2_1} and~\eqref{eq:L2_1.5},
we have
\begin{align*}
\int_{\gr(\phi|_A)}\frac{|\nabla^\phi\phi|^2}{1+|\nabla^\phi\phi|^2}\ d\mu_{E_\phi}&=\delta(n)\int_{\gr(\phi|_A)}\frac{|\nabla^\phi\phi|^2}{1+|\nabla^\phi\phi|^2}\ d\haush^{2n+1}=\\
	&\hspace*{-4cm}=\delta(n)\bigg(\int_{\gr(\phi|_A)\cap\de E\cap A*I}\frac{|\nabla^\phi\phi|^2}{1+|\nabla^\phi\phi|^2}\ d\haush^{2n+1}+\int_{(\gr(\phi|_A)\setminus\de E)\cap A*I}\frac{|\nabla^\phi\phi|^2}{1+|\nabla^\phi\phi|^2}\ d\haush^{2n+1}\bigg)\\
	&\hspace*{-4cm}\le 2\int_{\de E\cap A*I}\frac{|\nu_E-\nu|_g^2}{2}\ d\mu_E+\int_{(\gr(\phi|_A)\setminus\de E)\cap A*I}\frac{|\nabla^\phi\phi|^2}{1+|\nabla^\phi\phi|^2}\ d \mu_{E_\phi},
\end{align*} 
where $\delta(n)=\tfrac{2\omega_{2n-1}}{\omega_{2n+1}}$. Let $\mu$ be  the non-negative measure on $D_{\frac{k_2}{5124}}$  defined as
\begin{equation}\label{eq:def_mu_alpha}
\mu(A)=2\int_{\de E\cap A*I}\frac{|\nu_E-\nu|_g^2}{2}\ d \mu_E +\int_{(\gr(\phi|_A)\setminus\de E)\cap A*I}\frac{|\nabla^\phi\phi|^2}{1+|\nabla^\phi\phi|^2}\ d \mu_{E_\phi} ,
\end{equation} 
for any Borel set $A\subset D_{\frac{k_2}{5124}}$, where $\nu=-X_1$ as usual.

Let $0<\eta<1$ be a number that will be fixed later. We let 
\begin{equation*}
K_\eta=\big\{x\in D_{\frac{k_2}{5124}} : M\mu(x)\le\eta\big\},
\end{equation*}
where $M\mu$ is the local maximal function of $\mu$ defined in~\eqref{eq:def_loc_max_func} with $s=\tfrac{k_2}{20496}$. We assume $k_2>20496$ and we define 
\begin{equation*}
K=K_\eta\cap D_1.
\end{equation*}

We now prove~\eqref{eq:lip_alpha_easy_1}. Since $\phi$ is a $1$-representative
of Theorem~\ref{th:lip_app_easy} (with the scaling factor~$\frac{k_2}{5124}$),
by Remark~\ref{remark:sigma_repres} it is enough to prove that
$K\in\mathcal{A}(1,\frac{k_2}{5124})$. To this end, let us fix $p\in M\cap
D_1*I$ and $q\in M\cap K*I$. We proceed as in Step~1 of  the proof of
Theorem~\ref{th:lip_app_easy}. Indeed, by~\cite{monti&vittone15}*{Lemma~3.3}, we
have
\begin{equation}\label{eq:small_hgt_eta}
|\hgt(\xi)|<1 \qquad \text{for all } \xi\in C_{\frac{k_2}{5124}}\cap\de E,
\end{equation}
since $E$ is a $(\tfrac{1}{k_2},k_2)$-minimizer of $H$-perimeter in $C_{\frac{k_2}{2562}}$ and, by the scaling property of the excess, we can estimate
\begin{equation*}
\mathbf{e}(\tfrac{k_2}{2562})\le 2562^{2n+1}\mathbf{e}(k_2)\le 2562^{2n+1}\eps_2(n,\alpha)\le\omega(n,\tfrac{1}{2},\tfrac{1}{k_2},k_2),
\end{equation*} 
provided we assume
\begin{equation*}
\eps_2(n,\alpha)\le 2562^{-2n-1}\omega(n,\tfrac{1}{2},\tfrac{1}{k_2},k_2).
\end{equation*}
Here, as in the proof of Theorem~\ref{th:lip_app_easy},
$\omega(n,t,\Lambda,r_0)$, with $t\in(0,1)$, is the constant given
in~\cite{monti&vittone15}*{Lemma~3.3}.  
Thus we have $p,q\in C_1$ and $d_C(p,q)<8$, where $d_C$ is the quasi-distance
given by the quasi-norm $\|\cdot\|_C$ defined in~\eqref{eq:C_quasi_norm}.
Moreover, $q=\pi(q)*\hgt(q)\mathrm{e}_1$ with $\pi(q)\in K$ and $|\hgt(q)|<1$.
Since 
\begin{equation}\label{eq:proj_cyl_in_disk}
C_s(\xi)\subset\pi(C_s(\xi))*(-s-\hgt(\xi),\hgt(\xi)+s)\subset D_{2s}(\pi(\xi))*I
\end{equation}
for any $\xi\in C_1$ and $0<s<\tfrac{k_2}{5124}-1$, we can estimate
\begin{align*}
\mathbf{e}(q,s)&=\frac{1}{s^{2n+1}}\int_{C_s(q)\cap\de E} \frac{|\nu_E-\nu|_g^2}{2}\ d \mu_E  \\
	&\le\frac{1}{s^{2n+1}}\int_{\de E\cap D_{2s}(\pi(q))*I} \frac{|\nu_E-\nu|_g^2}{2}\ d \mu_E  \\
	&\le 2^{2n+1}\kappa_n\sup_{0<\rho<\frac{k_2}{5124}-\|\pi(q)\|_\infty}\frac{1}{\kappa_n \rho^{2n+1}}\int_{\de E\cap D_\rho(\pi(q))*I} \frac{|\nu_E-\nu|_g^2}{2}\ d \mu_E \\
	&\le 2^{2n}\kappa_n M\mu(\pi(q))\le 2^{2n}\kappa_n\eta
\end{align*} 
for any $0<s<\tfrac{k_2}{10248}$, where $\kappa_n=\leb^{2n}(D_1)$ as in~\eqref{eq:leb_meas_disk}.

We consider the blow-up of $E$ at scale $d_C(p,q)$ centered at $q$, that is, $F=E_{q,d_C(p,q)}$. By Remark~\ref{remark:scaling_perim_min}, $F$ is a $(\Lambda'',r''_0)$-perimeter minimizer in $(C_{k_2})_{q,d_C(p,q)}$, with
\begin{equation*}
\Lambda''=\Lambda'\, d_C(p,q), \qquad r_0''=\frac{r_0'}{d_C(p,q)}>1. 
\end{equation*}
Now
\begin{equation*}
C_{16}\subset (C_{k_2})_{q,d_C(p,q)}, \qquad \Lambda''r_0''\le1, \qquad 0\in\de F 
\end{equation*}
and, by the scaling property of the excess and by definition of $M_0$,
\begin{equation*}
\mathbf{e}(F,0,16,\nu)=\mathbf{e}(E,q,16 d_C(p,q),\nu)\le2^{2n}\kappa_n\eta,
\end{equation*}
since we can choose $k_2>1311744$. Therefore, provided we assume
\begin{equation*}
2^{2n}\kappa_n\eta\le\eps_0(n),
\end{equation*}
by Theorem~\ref{th:H-height_bound} we have 
\begin{equation*}
\sup\big\{ |\hgt(\xi)| : \xi\in C_1\cap\de F\big\}\le C(n)\eta^{\frac{1}{2(2n+1)}},
\end{equation*} 
where $C(n)$ is a dimensional constant. In particular, choosing
\begin{equation*}
\xi=\frac{1}{d_C(p,q)}\,q^{-1}*p\in C_1\cap\de F,
\end{equation*}
we get
\begin{equation}\label{eq:est_lip_q_eta}
|\hgt(q^{-1}*p)|\le C(n)\eta^{\frac{1}{2(2n+1)}}d_C(p,q).
\end{equation}
We now set 
\begin{equation*}
L'(n,\eta)=C(n)\eta^{\frac{1}{2(2n+1)}}
\end{equation*}
and we choose $\eta$ so small that $L'(n,\eta)\le L(n)$, where $L(n)<1$ is as in~\eqref{eq:def_L(n)}. Then, by~\eqref{eq:est_lip_q_eta}, we conclude that $d_C(p,q)=\|\pi(q^{-1}*p)\|_\infty$ and we get
\begin{equation}\label{eq:ok_repres}
|\hgt(q^{-1}*p)|\le L(n)\|\pi(q^{-1}*p)\|_\infty \qquad\text{for all } p\in M\cap D_1*I,\ q\in M\cap K*I,
\end{equation} 
so $K\in\mathcal{A}(1,\tfrac{k_2}{5124})$. Thus, by~\eqref{eq:small_hgt_eta} and~\eqref{eq:ok_repres}, equality~\eqref{eq:lip_alpha_easy_1} follows.

\medskip

\textit{Step~2: proof of~\eqref{eq:lip_alpha_easy_3}}. 
We now apply Lemma~\ref{lemma:max_func_disks} with $s=\tfrac{k_2}{20496}$ and measure $\mu$ as defined in~\eqref{eq:def_mu_alpha}. By Theorem~\ref{th:lip_app_easy}, we have
\begin{align}\label{eq:check_est_mu_disk}
\mu(D_{\nicefrac{k_2}{5124}})& =2\int_{\de E\cap C_{\nicefrac{k_2}{5124}}}\frac{|\nu_E-\nu|_g^2}{2}\ d \mu_E +\int_{(\gr(\phi)\setminus\de E)\cap C_{\nicefrac{k_2}{5124}}}\frac{|\nabla^\phi\phi|^2}{1+|\nabla^\phi\phi|^2}\ d \mu_{E_\phi}  \nonumber \\
	&\le 2\left(\tfrac{k_2}{5124}\right)^{2n+1}\mathbf{e}(\tfrac{k_2}{5124})+C(n)\haush^{2n+1}\Big((\de E\bigtriangleup\gr(\phi))\cap C_{\nicefrac{k_2}{5124}}\Big)\le C'(n)\,\mathbf{e}(k_2),
\end{align}
where $C(n)$ and $C'(n)$ are dimensional constants. We now choose $\eta=\mathbf{e}(k_2)^{2\alpha}$. In order to apply Lemma~\ref{lemma:max_func_disks}, we need to check that
\begin{equation*}
\mu(D_{\nicefrac{k_2}{5124}})\le\dfrac{\eta}{5^{2n+1}}\kappa_n \left(\frac{k_2}{20496}\right)^{2n+1}.
\end{equation*}
By~\eqref{eq:check_est_mu_disk}, this follows if we assume that
\begin{equation*}
\eps_2(n,\alpha)\le\left(\frac{\kappa_n }{C'(n)}\left(\frac{k_2}{102480}\right)^{2n+1}\right)^{\frac{1}{1-2\alpha}}.
\end{equation*}
This condition on $\eps_2(n,\alpha)$ is the only one that depends also on the parameter $\alpha$. Thus, by~\eqref{eq:meas_J_theta} in Lemma~\ref{lemma:max_func_disks} and by~\eqref{eq:check_est_mu_disk}, we conclude that
\begin{align*}
\leb^{2n}(D_1\setminus K)=\leb^{2n}(J_\eta\cap D_1)&\le\dfrac{5^{2n+1}}{\eta}\,\mu\left(J_{\eta/2^{2n+1}}\cap D_{1+\frac{k_2}{102480}}\right)\\
	&\le\dfrac{5^{2n+1}}{\mathbf{e}(k_2)^{2\alpha}}\,\mu\left(D_{\nicefrac{k_2}{5124}}\right)\le 5^{2n+1}C'(n)\mathbf{e}(k_2)^{1-2\alpha},
\end{align*}
which proves~\eqref{eq:lip_alpha_easy_3}.

\medskip

\textit{Step~3: proof of~\eqref{eq:lip_alpha_easy_2}}. By Lemma~\ref{lemma:BV_est} and by~\cite{cittietal14}*{Proposition~4.4}, we have
\begin{align*}
\mu_\phi(A)^2&=\left(\int_A|\nabla^\phi\phi| \ d\leb^{2n}\right)^2\\
	&\le\sqrt{1+\|\nabla^\phi\phi\|^2_{L^\infty(D_{\nicefrac{k_2}{5124}})}}\,
\leb^{2n}(A)\int_{\gr(\phi|_A)}\frac{|\nabla^\phi\phi|^2}{1+|\nabla^\phi\phi|^2}\ d \mu_{E_\phi} \\
	&\le C(n)\,\leb^{2n}(A)\int_{\gr(\phi|_A)}\frac{|\nabla^\phi\phi|^2}{1+|\nabla^\phi\phi|^2}\ d \mu_{E_\phi}
\end{align*}
for all Borel sets $A\subset D_1$, where $C(n)$ is a dimensional constant.
Moreover, for any $x\in K$ and $8r<\tfrac{k_2}{5124}-\|x\|_\infty$,
by~\eqref{eq:comp_d_phi_graph} in Lemma~\ref{lemma:comp_d_phi_graph},
by~\eqref{eq:D_and_C_comparison}  and by~\eqref{eq:proj_cyl_in_disk}, we
have 
\begin{align*}
\int_{\Phi(U_\phi(x,r))}\frac{|\nabla^\phi\phi|^2}{1+|\nabla^\phi\phi|^2}\ d \mu_{E_\phi} &\le\int_{\Gamma\cap B_{2r}(\Phi(x))}\frac{|\nabla^\phi\phi|^2}{1+|\nabla^\phi\phi|^2}\ d\mu_{E_\phi}\\
	&\hspace*{-3cm}\le\int_{\Gamma\cap C_{4r}(\Phi(x))}\frac{|\nabla^\phi\phi|^2}{1+|\nabla^\phi\phi|^2}\ d \mu_{E_\phi}\\
	&\hspace*{-3cm}\le 2\int_{M\cap D_{8r}(x)*I}\frac{|\nu_E-\nu|_g^2}{2}\ d \mu_E +\int_{(\Gamma\setminus M)\cap D_{8r}(x)*I}\frac{|\nabla^\phi\phi|^2}{1+|\nabla^\phi\phi|^2}\ d \mu_{E_\phi} \\
	&\hspace*{-3cm}=\mu(D_{8r}(x)).
\end{align*}
Therefore, for any $x\in K$ and $8r<\tfrac{k_2}{5124}-\| x\|_\infty$, we get
\begin{equation}\label{eq:square_measure}
\mu_\phi(U_\phi(x,r))^2\le C(n)\,\leb^{2n}(U_\phi(x,r))\,\mu(D_{8r}(x)).
\end{equation}
We now apply Lemma~\ref{lemma:phi_max_func}. We choose the parameter $s>0$ in Lemma~\ref{lemma:phi_max_func} such that
\begin{equation*}
D_1\subset U_\phi(0,s) \quad \textrm{and} \quad U_\phi(0,\rho(n)s)\subset D_{k_2},
\end{equation*}
where $\rho(n)$ is the dimensional constant defined in~\eqref{eq:def_rho}. 
 Since \mbox{$\Lip_H(\phi)\le L(n)<1$},
where $L(n)$ is the dimensional constant defined in~\eqref{eq:def_L(n)},
possibly choosing $\eps_2(n,\alpha)$ smaller, we can directly assume that
$L(n)\le\ell(n)$ as in~\eqref{eq:def_ell}. In particular, the constant
$c(n,\Lip_H(\phi))$ appearing in~\eqref{eq:phi_theta_lip} of
Lemma~\ref{lemma:phi_max_func}, is controlled from above by a dimensional
constant.  Since $\sup_\W|\phi|<1$, by Lemma~\ref{lemma:U_phi_comp_disk} we can
choose $s=3$ provided that we also choose
\begin{equation*}
k_2(n)\ge 3\rho(n)+2\sqrt{3\rho(n)}.
\end{equation*}
We then have
\begin{equation*}
r_\phi(x,3)=\frac{3\rho(n)}{c_L}-d_\phi(x,0)\le 3\rho(n),
\end{equation*}
where $r_\phi(x,s)$ was defined in~\eqref{eq:def_r_phi(x,s)}.
By~\eqref{eq:square_measure} and~\eqref{eq:meas_phi_balls},  for any $x\in K$ we
have 
\begin{align*}
[\mu_\phi](x)^2&=\sup_{0<r<r_\phi(x,3)}
\dfrac{\mu_\phi(U_\phi(x,r))^2}{\leb^{2n}(U_\phi(x,r))^2}\le\,C(n)\sup_{
0<r<3\rho(n)}\dfrac{\mu(D_{8r}(x))}{\leb^{2n}(U_\phi(x,r))}\\
	&\le\dfrac{C(n)8^{2n+1}\kappa_n}{c_1^L}\sup_{0<r<3\rho(n)}\dfrac{\mu(D_{8r}(x))}{\kappa_n (8r)^{2n+1}}\\
	&\le C'(n)\sup_{0<\rho<24\rho(n)}\dfrac{\mu(D_\rho(x))}{\kappa_n \rho^{2n+1}}
\end{align*}
where $C'(n)$ is a dimensional constant. Now we can choose 
\begin{equation*}
k_2>122976\rho(n)+5124,
\end{equation*}
so that $24\rho(n)\le\tfrac{k_2}{5124}-\|x\|_\infty$ for any $x\in D_1$. Therefore, for any $x\in K$, we get
\begin{equation*}
[\mu_\phi](x)\le \sqrt{C'(n)\,\eta}=C''(n)\,\mathbf{e}(k_2)^\alpha,
\end{equation*} 
where $C''(n)$ is a positive dimensional constant. Thus $K\subset U_\phi(0,3)\setminus J^\phi_\theta$, where $J^\phi_\theta$ is as in~\eqref{eq:def_J_theta_phi} and $\theta=C''(n)\,\mathbf{e}(k_2)^\alpha$. Therefore, by~\eqref{eq:phi_theta_lip} in Lemma~\ref{lemma:phi_max_func}, we conclude that
for all $x,y\in K$ we have
\begin{equation*}
|\phi(x)-\phi(y)|\le C(n)\,\mathbf{e}(k_2)^\alpha\, d_\phi(x,y) .
\end{equation*}
This proves~\eqref{eq:lip_alpha_easy_2} and the proof of Theorem~\ref{th:lip_alpha_easy} is complete.
\end{proof}
 
Theorem~\ref{th:lip_alpha_easy} leads to the following result, which contains
Theorem~\ref{th_intro:alpha_app} in the Introduction as a particular case.

\begin{corollary}\label{coroll:alpha_app_easy}
Let $n\geq 2$ and $\alpha\in(0,\frac{1}{2})$. There exist positive constants $C_3(n)$, $\eps_3(\alpha,n)$ and $k_3=k_3(n)$ with the following property. 
 For any set $E\subset\He^n$ that is a $(\Lambda,r_0)$-minimizer of
$H$-perimeter in $C_{k_3}$ with $\mathbf{e}(k_3)\le\eps_3(\alpha,n)$, $\Lambda
r_0\le1$, $r_0>k_3$ and $ 0\in \partial E$, 
there exist a set $K\subset D_1$ and an intrinsic Lipschitz function $\phi\colon\W\to\R$ such that:
\begin{equation*}
\leb^{2n}(D_1\setminus K)\le C_3(n)\,\mathbf{e}(k_3)^{1-2\alpha},
\end{equation*}
\begin{equation}\label{eq:alpha_app_easy_2}
\gr(\phi|_K)=\de E\cap K*(-1,1),
\end{equation}
\begin{equation}\label{eq:alpha_app_easy_3}
\haush^{2n+1}\big((\de E\bigtriangleup\gr(\phi))\cap C_1\big)\le C_3(n)\,\mathbf{e}(k_3)^{1-2\alpha},
\end{equation}
\begin{equation*}
\Lip_H(\phi)\le C_3(n)\,\mathbf{e}(k_3)^\alpha,
\end{equation*}
\begin{equation}\label{eq:alpha_app_easy_5}
\int_{D_1}|\nabla^\phi\phi|^2\ d\leb^{2n}\le C_3(n)\,\mathbf{e}(k_3).
\end{equation}
\end{corollary}

\begin{proof}
Let $\alpha\in(0,\frac{1}{2})$ be fixed and assume that $\eps_3(n,\alpha)\le\eps_2(n,\alpha)$ and $k_3=k_2$. Let $K$ and $\phi$ be as in Theorem~\ref{th:lip_alpha_easy}. Recall that, by construction, $\Lip_H(\phi)<1$ and $\sup_\W|\phi|<1$. Moreover, by~\eqref{eq:lip_alpha_easy_2}, we have
\begin{equation*}
\Lip_H(\phi|_K)\le C_2(n)\,\mathbf{e}(k_2)^\alpha.
\end{equation*}
Thus, according to Proposition~\ref{prop:intrinsic_lip_ext}, choosing $\eps_3(n,\alpha)\le\eps_2(n,\alpha)$ sufficiently small, we can extend $\phi$ outside $K$ to the whole $\W$ in such a way that $\sup_\W|\phi|<1$ and
\begin{equation*}
\Lip_H(\phi)\le C(n)\,\mathbf{e}(k_3)^\alpha,
\end{equation*}
where $C(n)$ is a dimensional constant. Thus we only need to prove~\eqref{eq:alpha_app_easy_3} and~\eqref{eq:alpha_app_easy_5}.

We prove~\eqref{eq:alpha_app_easy_3}. Let $J=D_1\setminus K$, $I=(-1,1)$, and note that, by~\eqref{eq:alpha_app_easy_2}, we have
\begin{align*}
\haush^{2n+1}\big((\de E\bigtriangleup\gr(\phi))\cap C_1\big)& 
=\haush^{2n+1}\big((\de E\setminus\gr(\phi))\cap
J*I\big)
\\
&
\qquad 
\qquad 
+\haush^{2n+1}\big((\gr(\phi)\setminus\de E)\cap J*I\big)
\\
	&
\le\haush^{2n+1}(\de E\cap J*I)+\haush^{2n+1}(\gr(\phi)\cap
J*I).
\end{align*}
On the one hand, by definition of excess and by~(3.56)
in~\cite{monti&vittone15}*{Lemma~3.4}, 
we have
\begin{align}\label{eq:alpha_3.1}
\haush^{2n+1}(\de E\cap J*I)
&=\int_{\de E\cap J*I}1+\langle \nu_E,X_1\rangle_g\ d\haush^{2n+1}-\int_{\de
E\cap J*I}\langle \nu_E,X_1\rangle_g\ d\haush^{2n+1}
=\nonumber\\
	&=\delta(n)^{-1}\int_{\de E\cap J*I}\frac{|\nu_E-\nu|_g^2}{2}\ d \mu_E
+\leb^{2n}(J)\nonumber
\\
	&\le\delta(n)^{-1}\mathbf{e}(1)+\leb^{2n}(J),
\end{align}
thus, by the scaling property of the excess and by~\eqref{eq:lip_alpha_easy_3}, we can estimate
\begin{equation}\label{eq:alpha_3.2}
\haush^{2n+1}(\de E\cap J*I)\le \delta(n)^{-1}\,k_3^{2n+1}\,\mathbf{e}(k_3)+C_2(n)\,\mathbf{e}(k_3)^{1-2\alpha}\le C(n)\,\mathbf{e}(k_3)^{1-2\alpha},
\end{equation}
where $C(n)$ is a dimensional constant. On the other hand, by the area
formula~\eqref{eq:area_formula}, we have
\begin{align}\label{eq:alpha_3.3}
\haush^{2n+1}(\gr(\phi)\cap J*I)&=\delta(n)^{-1}\int_J\sqrt{1+|\nabla^\phi\phi|^2}\ d\leb^{2n}\nonumber\\
	&\le\delta(n)^{-1}\,\sqrt{1+\|\nabla^\phi\phi\|_{L^\infty(D_1)}^2}\,\leb^{2n}(J),
\end{align} 
and thus, by~\cite{cittietal14}*{Proposition~4.4} and again by~\eqref{eq:lip_alpha_easy_3}, we can estimate
\begin{equation*}
\haush^{2n+1}(\gr(\phi)\cap J*I)\le C(n)\,\mathbf{e}(k_3)^{1-2\alpha},
\end{equation*}
where $C(n)$ is a dimensional constant. Combining~\eqref{eq:alpha_3.1} with~\eqref{eq:alpha_3.2} and~\eqref{eq:alpha_3.3}, we prove~\eqref{eq:alpha_app_easy_3}.

Finally, we prove~\eqref{eq:alpha_app_easy_5}. Since $D_1=K\cup J$ with disjoint union, we can split
\begin{equation}\label{eq:alpha_5.1}
\int_{D_1}|\nabla^\phi\phi|^2\ d\leb^{2n}=\int_K|\nabla^\phi\phi|^2\
d\leb^{2n}+\int_J|\nabla^\phi\phi|^2\ d\leb^{2n}.
\end{equation}
On the one hand, by~\cite{cittietal14}*{Proposition~4.4} and
by~\eqref{eq:lip_alpha_easy_1}, we have
\begin{align}
\label{eq:alpha_5.2}
\int_K|\nabla^\phi\phi|^2\ d\leb^{2n}
&=\int_{\gr(\phi|_K)}
\frac{|\nabla^\phi\phi|^2}{\sqrt{1+|\nabla^\phi\phi|^2}}\ d
\mu_{E_\phi}
\nonumber
\\
	&\le\sqrt{1+\|\nabla^\phi\phi\|_{L^\infty(D_1)}^2}\int_{\gr(\phi|_K)}
\frac{|\nabla^\phi\phi|^2}{1+|\nabla^\phi\phi|^2}\ d \mu_{E_\phi} 
\nonumber
\\
	&\le C(n)\int_{M\cap K*I}\frac{|\nu_E-\nu|_g^2}{2} \ d \mu_E 
\le
C(n)\,\mathbf{e}(1)\le C'(n)\,\mathbf{e}(k_3),
\end{align}
where $C(n)$ and $C'(n)$ are dimensional constants. On the other hand, again by~\cite{cittietal14}*{Proposition~4.4} and by~\eqref{eq:lip_alpha_easy_2}, we have
\begin{align}\label{eq:alpha_5.3}
\int_J|\nabla^\phi\phi|^2\ d\leb^{2n}&\le\|\nabla^\phi\phi\|_{L^\infty(D_1)}^2\,\leb^{2n}(J)\nonumber\\
	&\le C(n)\Lip_H(\phi)^2\leb^{2n}(J)\le C'(n)\,\mathbf{e}(k_3).
\end{align}
Combining~\eqref{eq:alpha_5.1} with~\eqref{eq:alpha_5.2} and~\eqref{eq:alpha_5.3}, we prove~\eqref{eq:alpha_app_easy_5}.
\end{proof}


\bigskip

\end{document}